\documentclass[runningheads]{llncs}
\usepackage{graphicx}
 \usepackage{hyperref}

\usepackage{tikz}
\usepackage{subcaption}

\def\--{\mbox{--}}

\newcommand{\B}[1]{B$_{#1}$-VPG}
\newcommand{\comp}[1]{\overline{#1}}
\usepackage{amsmath, amssymb}
\newcommand{\p}[2]{\ensuremath{\prec_D^{{#1},{#2}}}}
\def\n{2}
\usepackage{algorithmic,algorithm}

\definecolor{darkseagreen}{rgb}{0.56, 0.74, 0.56}
\definecolor{forestgreen}{rgb}{0.13, 0.55, 0.13}
\definecolor{darkblue}{rgb}{0, 0, 102}
\definecolor{lightblueArrow}{rgb}{51, 51, 255}

\usepackage[title]{appendix}
\usepackage{lineno}
\usepackage{enumitem}

\definecolor{lime}{HTML}{A6CE39}
\DeclareRobustCommand{\orcidicon}{
	\begin{tikzpicture}
	\draw[lime, fill=lime] (0,0) circle [radius=0.16] 
	node[white] {{\fontfamily{qag}\selectfont \tiny ID}};
	\draw[white, fill=white] (-0.0625,0.095) 
	circle [radius=0.007];
	\end{tikzpicture}
}
\foreach \x in {DR, SKP}{%
	\expandafter\xdef\csname orcid\x\endcsname{\noexpand\href{https://orcid.org/\csname orcidauthor\x\endcsname}{\noexpand\orcidicon}}
}

\begin{document}
	\title{Characterization and a 2D Visualization of \B0 Cocomparability Graphs}
	\titlerunning{Characterization \& a 2D Visualization of \B0 Cocomparability Graphs}
	%
	\author{Sreejith K. Pallathumadam\orcidSKP \and
		Deepak Rajendraprasad\orcidDR}
	
	\authorrunning{S. K. Pallathumadam and D. Rajendraprasad}
%
%
\institute{Indian Institute of Technology Palakkad, India}

\maketitle              
%
	\begin{abstract}
		\emph{\B0 graphs} are intersection graphs of vertical and horizontal
		line segments on a plane.
		Cohen, Golumbic, Trotter, and Wang [Order, 2016] pose the question of
		characterizing \B0 permutation graphs.
		We respond here by characterizing \B0 cocomparability graphs. 
		This characterization also leads to a polynomial time recognition 
		and \B0 drawing algorithm for the class.
		Our \B0 drawing algorithm starts by fixing any one of the many posets $P$
		whose cocomparability graph is the input graph $G$. 
		The drawing we obtain not only visualizes  $G$ in that one can distinguish
		comparable pairs from incomparable ones, but one can also identify which
		among a comparable pair is larger in $P$ from this visualization.
	
	\keywords{Poset Visualization \and Permutation graph  \and Cocomparability graph \and \B0 \and Graph drawing.}
\end{abstract}
\section{Introduction}
Representing a graph as an intersection graph of two-dimensional geometric
objects like strings, line segments, rectangles and disks is a means to depict
a graph on the plane. When the graph being represented is a comparability or
cocomparability graph, one can also ask whether the ``direction'' of the
comparability relation in the associated poset can also be inferred from the
drawing. In this paper we characterize cocomparability graphs which can be
represented as intersection graphs of vertical and horizontal line segments in
a plane. For the posets whose cocomparability graphs can be represented thus,
we describe a representation from which one can also infer the direction of the
comparability relation. Our drawing algorithm runs in polynomial time.

\B{k} graphs are intersection graphs of simple paths with at most $k$ bends
on a two-dimensional grid.
Here, a path is simple if it does not pass through any grid vertex twice,
and two paths are said to intersect if they share a vertex of the grid.
The name \B{k} is an abbreviation for Vertex-intersection graphs of
$k$-Bend Paths on a Grid.
In particular, \B0 graphs are intersection graphs of vertical and
horizontal line segments on a plane. 
The \emph{bend number} of a graph $G$ is the minimum $k$ for which $G$
belongs to \B{k}. 

The \emph{dimension} of a poset $P = (X, \prec)$ is the smallest $k$ such
that $\prec$ is the intersection of $k$ total orders on $X$.
The \emph{comparability graph of $P$} is the undirected graph on the vertex
set $X$ with edges between the pairs of elements comparable in $P$. 
A graph $G$ is a \emph{comparability graph} if it is the comparability
graph of a poset.
If two posets have the same comparability graph, then they have the same
dimension \cite{trotter1976dimension}.
Hence we can unambiguously define the \emph{dimension of a comparability graph} $G$ as the dimension of any poset $P$ whose comparability graph is
$G$. 
The complement of a comparability graph is a \emph{cocomparability} graph.
A \emph{permutation graph} is a comparability graph of dimension at most two.
It is known that a graph $G$ is a permutation graph if and only if $G$
is both comparability and cocomparability \cite{pnueli1971transitive}.

Cohen et al. \cite{cohen2016posets} illustrated, via an elegant
picture-proof, that if $G$ is a comparability graph of dimension $k$ ($k
\geq 1$), the bend number of its complement $\comp{G}$ is at most $k-1$.
In particular therefore, the bend number of a permutation graph is either
$0$ or $1$.
They posed the problem of characterizing permutation graphs with bend
number $0$ as an open question (Qn $4.2$ in \cite{cohen2016posets}).
We settle this question with a stronger result. 
We characterize cocomparability graphs with bend number $0$ as follows
(Theorem~\ref{theorem:vpg}). 

The simple cycle on $k$ vertices is denoted by $C_k$. 
A $C_4$ together with an additional edge $e$ between two non-consecutive
vertices of the $C_4$ is a \emph{diamond} and the edge $e$ is a
\emph{diamond diagonal}.
Two vertices $x$ and $y$ in a graph $G$ are \emph{diamond related} if there
exists a path from $x$ to $y$ in $G$ made up of diamond diagonals alone.
This is easily verified to be an equivalence relation that refines the
connectivity relation in $G$.

\begin{theorem}
	\label{theorem:vpg}
	A cocomparability graph $G$ is \B0 if and only if 
	\begin{enumerate}[label=(\roman*)]
		\item No two vertices of an induced $C_4$ in $G$ are diamond
		related, and 
		\item $G$ does not contain an induced subgraph isomorphic to
		$\comp{C_6}$, the complement of $C_6$.
	\end{enumerate}
\end{theorem}

A poset $P = (X, \prec)$ is an \emph{interval order} if all the elements of
$X$ can be mapped to intervals on $\mathbb{R}$ such that $\forall x,y \in X$, $x
\prec y$ if and only if the interval representing $x$ is disjoint from
and to the left of the interval representing $y$. 
Complements of the comparability graphs of interval orders form the well
known class of \emph{interval graphs}.
While interval graphs are trivially \B0, it is known that there exists
interval orders of arbitrarily high dimension \cite{bogart1976bound}.
Hence the class of \B0 cocomparability graphs is richer than the class 
of \B0 permutation graphs. 
In fact, since permutation graphs are $\comp{C_6}$-free, the first of the
two conditions in Theorem~\ref{theorem:vpg} characterizes \B0 permutation
graphs.

\begin{corollary}
	\label{cor:perm}
	A permutation graph $G$ is \B0 if and only if no two vertices of an
	induced $C_4$ in $G$ are diamond related.
\end{corollary}

A naive check for the conditions in Theorem~\ref{theorem:vpg} can be done
in $O(n^6)$ time. 
Combining this with any of the known polynomial time recognition algorithms
for cocomparability graphs \cite{golumbic1977complexity,corneil1999lbfs} will give a polynomial time recognition algorithm for \B0 cocomparability graphs.
We do not try to optimize the recognition algorithm here, but only note
that this is in contrast to the NP-completeness of recognizing \B0 graphs.

The above algorithm starts by fixing a partial order $P_G$ whose cocomparability
graph is $G$ and a linear extension $\sigma$ of $P_G$. The resulting drawing $D$
ends up being a representation of $P_G = (V(G), \prec_P)$ in the following sense.
We define three binary relations $\prec_D^{v,h}$, $\prec_D^{h,v}$ and $\prec_D$
among vertices of $G$ based on the drawing $D$ as follows. 

\begin{definition}
	\label{defnDiagramOrder}
	Let $x$ and $y$ be two vertices in $V(G)$ and let $I_x$ and $I_y$ be the line
	segments representing them in $D$, respectively. 
	\begin{itemize}
		\item 
		$x \p{v}{h} y$ if  $I_x$ is either vertically below $I_y$ or if both are
		intersected by a horizontal line, $I_x$ is to the left of $I_y$.	
		\item
		$x \p{h}{v} y$ if $I_x$ is either to the left of $I_y$ or if both are
		intersected by a vertical line, $I_x$ is vertically below $I_y$. 
		\item
		$x \prec_D y$ if and only if $x \p{v}{h} y$ when $I_x$ is horizontal and
		$x \p{h}{v} y$ when $I_x$ is vertical.
	\end{itemize}
\end{definition}

While the above relations are not even partial orders in general, in our
drawing $D$, the relation $\prec_D$ faithfully captures $\prec_P$.
Theorem~\ref{theorem:2D} states this formally and Figure~\ref{fig:vpgDrawingEg}
illustrates an example.

\begin{theorem}
	\label{theorem:2D}
	Any poset $P_G = (V_G,\prec_P)$ whose cocomparability graph $G$ is \B0 has a two dimensional visualization $D$ such that $x \prec_P y$ if and only if $x \prec_D y$.
\end{theorem}
\begin{figure}
	\centering 
	\begin{subfigure}[b]{0.44\linewidth}
		
		\centering
		\begin{tikzpicture}[scale=.7]
		\node (a) at (0,0) {\textbf{\color{forestgreen}a}};
		\node [right of=a, node distance=.7in] (b) {\textbf{\color{darkblue}b}};
		\node [above of=a, node distance=.5in] (d) {\textbf{\color{forestgreen}d}};
		\node [above of=b, node distance=.5in] (e) {\textbf{\color{darkblue}e}};
		\node [left of=d, node distance=.5in] (c) {\textbf{\color{darkblue}c}};
		\node [right of=e, node distance=.5in] (f) {\textbf{\color{forestgreen}f}};
		
		\node [above of=e, node distance=.5in] (h) {\textbf{\color{forestgreen}h}};
		\node [left of=h, node distance=.7in] (g) {\textbf{\color{darkblue}g}};
		\node [right of=h, node distance=.6in] (i) {\textbf{\color{forestgreen}i}};
		
		\draw [darkseagreen, thick, shorten <= -2pt, shorten >=-2pt] (a) -- (c);
		\draw [darkseagreen, thick, shorten <= -2pt, shorten >=-2pt] (a) -- (d);
		\draw [darkseagreen, thick, shorten <= -2pt, shorten >=-2pt] (a) -- (f);
		\draw [blue, thick, shorten <= -2pt, shorten >=-2pt] (b) -- (c);
		\draw [blue, thick, shorten <= -2pt, shorten >=-2pt] (b) -- (e);
		\draw [blue, thick, shorten <= -2pt, shorten >=-2pt] (b) -- (f);
		
		\draw [blue, thick, shorten <= -2pt, shorten >=-2pt] (c) -- (g);
		\draw [blue, thick, shorten <= -2pt, shorten >=-2pt] (c) -- (h);
		
		\draw [darkseagreen, thick, shorten <= -2pt, shorten >=-2pt] (d) -- (g);
		\draw [darkseagreen, thick, shorten <= -2pt, shorten >=-2pt] (d) -- (h);
		\draw [darkseagreen, thick, shorten <= -2pt, shorten >=-2pt] (d) -- (i);
		
		\draw [blue, thick, shorten <= -2pt, shorten >=-2pt] (e) -- (g);
		\draw [blue, thick, shorten <= -2pt, shorten >=-2pt] (e) -- (h);
		\draw [blue, thick, shorten <= -2pt, shorten >=-2pt] (e) -- (i);
		
		\draw [darkseagreen, thick, shorten <= -2pt, shorten >=-2pt] (f) -- (h);
		\draw [darkseagreen, thick, shorten <= -2pt, shorten >=-2pt] (f) -- (i);
		\end{tikzpicture}
	\end{subfigure}
	\begin{subfigure}[b]{0.55
			\linewidth}
		\centering
		\begin{tikzpicture}[scale=.7]		
		
		\draw [darkblue, ultra thick, shorten <= -2pt, shorten >=-2pt] (2.4*\n,1.2*\n) -- (2.4*\n,3.6*\n);
		\node (pe) at (2.5*\n,1.1*\n) {\tiny $I_b$};
		\draw [darkblue, ultra thick, shorten <= -2pt, shorten >=-2pt] (4.1*\n,1.2*\n) -- (4.1*\n,3.6*\n);
		\node (pe) at (4.2*\n,1.1*\n) {\tiny $I_e$};
		\draw [darkblue, ultra thick, shorten <= -2pt, shorten >=-2pt] (4*\n,2.8*\n) -- (4*\n,5.2*\n);
		\node (pg) at (4.15*\n,5.2*\n) {\tiny $I_c$};
		\draw [darkblue, ultra thick, shorten <= -2pt, shorten >=-2pt] (5.6*\n,2.8*\n) -- (5.6*\n,5.2*\n);
		\node (pd) at (5.7*\n,2.7*\n) {\tiny $I_g$};
		\draw [forestgreen, ultra thick, shorten <= -2pt, shorten >=-2pt] (2*\n,1.6*\n) -- (4.4*\n,1.6*\n);
		\node (pd) at (1.9*\n,1.5*\n) {\tiny $I_a$};
		\draw [forestgreen, ultra thick, shorten <= -2pt, shorten >=-2pt] (2*\n,3.3*\n) -- (4.4*\n,3.3*\n);
		\node (pf) at (1.9*\n,3.2*\n) {\tiny $I_d$};
		\draw [forestgreen, ultra thick, shorten <= -2pt, shorten >=-2pt] (3.6*\n,3.2*\n) -- (6*\n,3.2*\n);
		\node (pd) at (6.1*\n,3.3*\n) {\tiny $I_f$};
		\draw [forestgreen, ultra thick, shorten <= -2pt, shorten >=-2pt] (3.6*\n,4.8*\n) -- (6.3*\n,4.8*\n);
		\node (pb) at (6.4*\n,4.7*\n) {\tiny $I_i$};
		\draw [forestgreen, ultra thick, shorten <= -2pt, shorten >=-2pt] (5.2*\n,4.9*\n) -- (6.3*\n,4.9*\n);
		\node (pb) at (6.4*\n,5*\n) {\tiny $I_h$};
		
		\draw [->][blue, thin]   (2.4*\n,3.2*\n) -- (3.54*\n,3.2*\n); 
		\draw [->][blue, thin]   (2.4*\n,3*\n) -- (3.95*\n,3*\n); 
		\draw [->][blue, thin]   (2.4*\n,2.9*\n) -- (4.1*\n,2.9*\n); 
		
		\draw [->][darkseagreen, thin]   (4*\n,1.6*\n) -- (4*\n,2.7*\n); 
		\draw [->][darkseagreen, thin]   (3.7*\n,1.6*\n) -- (3.7*\n,3.2*\n); 
		\draw [->][darkseagreen, thin]   (3.8*\n,1.6*\n) -- (3.8*\n,3.3*\n); 
		
		\draw [->][blue, thin]   (4*\n,4.9*\n) -- (5.1*\n,4.9*\n); 
		\draw [->][blue, thin]   (4*\n,5*\n) -- (5.6*\n,5*\n); 
		
		\draw [->][darkseagreen, thin]   (4.45*\n,3.3*\n) -- (5.6*\n,3.3*\n); 
		\draw [->][darkseagreen, thin]   (4.3*\n,3.3*\n) -- (4.3*\n,4.75*\n); 
		\draw [darkseagreen, thin]   (4.4*\n,3.3*\n) -- (4.4*\n,4.7*\n); 
		\draw [darkseagreen, thin]   (4.4*\n,4.7*\n) -- (5.4*\n,4.7*\n); 
		\draw [->][darkseagreen, thin]   (5.4*\n,4.7*\n) -- (5.4*\n,4.9*\n); 
		
		\draw [->][blue, thin]   (4.1*\n,3.1*\n) -- (5.6*\n,3.1*\n); 
		\draw [blue, thin]   (4.1*\n,3*\n) -- (6.2*\n,3*\n); 
		\draw [->][blue, thin]   (6.2*\n,3*\n) -- (6.2*\n,4.8*\n); 
		\draw [blue, thin]   (4.1*\n,2.9*\n) -- (6.3*\n,2.9*\n); 
		\draw [->][blue, thin]   (6.3*\n,2.9*\n) -- (6.3*\n,4.9*\n); 
		
		\draw [->][darkseagreen, thin]   (5.8*\n,3.2*\n) -- (5.8*\n,4.8*\n); 
		\draw [->][darkseagreen, thin]   (5.9*\n,3.2*\n) -- (5.9*\n,4.9*\n); 
		
		\end{tikzpicture}
	\end{subfigure}
	\caption{Hasse diagram of a poset corresponding to a \B0 cocomparability graph and a \B0 representation $D$ in which the covering relation is indicated by thin directed paths for clarity. The directed paths with blue color and green color respectively depict the relations \p{h}{v} and \p{v}{h}.}\label{fig:vpgDrawingEg}
\end{figure}
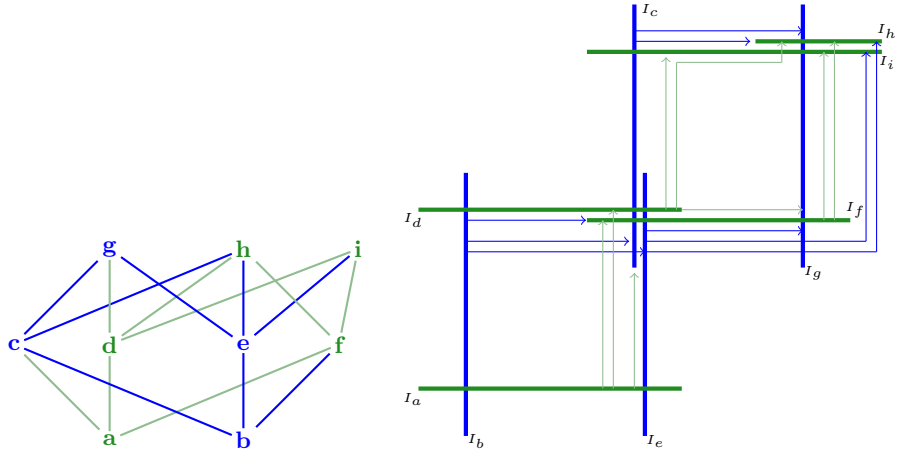

\subsection{Literature}
\B{k} graphs were introduced by Asinowski et al.\ in 2012
\cite{asinowski2012vertex} as a parameterized generalization for string
graphs (intersection graphs of curves in a plane) and grid intersection
graphs (bipartite graphs which are intersection graphs of vertical and
horizontal segments in the plane in which all the vertices in one part are
represented by vertical segments and all the vertices in the other part by
horizontal line segments).
Grid Intersection graphs (GIGs) are equivalent to bipartite \B0 graphs.
Similarly one can show that \B{k} graphs with an unrestricted $k$, which
are called VPG graphs simpliciter, are equivalent to string graphs.
\B0 graphs are also equivalent to $2$-DIR graphs, where a $k$-DIR graph
is an intersection graph of line segments lying in at most $k$ directions
in the plane. All these equivalences were formally established in
\cite{asinowski2012vertex}.

The NP-completeness of the recognition problem for VPG graphs follows from
that of string graphs \cite{kratochvil1991string,schaefer2003recognizing}. 
For \B0 graphs, it follows from that of $2$-DIR graphs
\cite{kratochvil1994intersection}.
Chaplick et al.\  showed that, $\forall k \geq 0$, it is NP-complete to
recognize whether a given graph $G$ is in \B{k} even when $G$ is
guaranteed to be in \B{k+1} and represented as such \cite{chaplick2012bend}.
This also shows that $\forall k \geq 0$ the classes \B{k} and \B{k+1} are
separated.
Cohen et al.\ showed that, $\forall k \geq 0$, there exists a
cocomparability graph with bend number $k$ (Theorem~3.1 in
\cite{cohen2016posets}). 
This shows that, $\forall k \geq 0$, the classes \B{k} and \B{k+1} are
separated within cocomparability graphs.
The question of a similar separation within chordal graphs was left
open in \cite{chaplick2012bend} and a partial answer was given in
\cite{chakraborty2019bounds}.

Since the \B{k} representation is a kind of planar representation, the bend
number of planar graphs have received special attention. 
Chaplick and Ueckerdt showed, disproving a conjecture in
\cite{asinowski2012vertex}, that every planar graph is \B2 \cite{chaplick2012planar}.
Every planar bipartite graph is a GIG \cite{hartman1991grid} and
hence \B0. 
The order dimension of GIGs has also been investigated in literature \cite{chaplick2018grid}.
A polynomial time decision algorithm for chordal \B0 graphs is developed in \cite{chaplick2011recognizing}.
Characterizations for \B0 are known within the classes of split
graphs, chordal bull-free graphs, chordal claw-free graphs
\cite{golumbic2013intersectionForDiamond} and block graphs
\cite{alcon2017vertex}.

Subclasses of cocomparability graphs within which a characterization
for \B0 is known include cographs, bipartite permutation graphs,
and interval graphs. 
Cographs, which form a subgraph of permutation graphs are \B0 if and
only if they do not contain an induced $W_4$
\cite{cohen2014characterizations}.  
A $W_4$ is a  $C_4$ together with 	a universal fifth vertex.
All bipartite permutation graphs are \B0 \cite{cohen2016posets}.
Interval graphs are trivially \B0, since the interval representation
itself is a \B0 representation.
Theorem~\ref{theorem:vpg} subsumes these three results.


Towards the end of this paper, we discuss a two dimensional visualization of posets. The most common way to visualize a poset $P = (X, \prec)$ so far is \emph{Hasse Diagram} (also called Order Diagram). 
The problem of drawing a Hasse diagram algorithmically was addressed by many algorithms \emph{e.g.} \emph{upward planar drawing} \cite{di1988algorithms}, \emph{dominance drawing} \cite{di1992area}, \emph{confluent drawing} \cite{dickerson2005confluent}, \emph{weak dominance drawing} \cite{kornaropoulos2012weak}. The key concern here is to get a crossing-free drawing in which no two upward edges cross at a non vertex point. 
The first three algorithms can only handle posets of dimension at most two and a few other cases.
Though our drawing handles only those posets whose cocomparability graph is \B0 irrespective of its dimension, crossing-freeness is not a concern in our drawing since comparability is inferred from the relative position of lines. 
\subsection{Terminology and Notation}

The complement of a graph $G$ is denoted as $\overline{G}$.
We denote a path and a cycle on $n$ vertices, respectively, by $P_n$ and
$C_n$.
A graph $G$ is said to be \emph{$H$-free} if $G$ contains no induced
subgraph isomorphic to the graph $H$.
A poset $P$ is said to be \emph{$T$-free} if $P$ contains no induced 
subposet isomorphic to the poset $T$. 
In this case, $P$ is also said to \emph{exclude} $T$.

The \emph{closed neighborhood} of a vertex $v$ is the set of neighbors of
$v$ together with $v$.
An \emph{Asteroidal Triple} (\emph{AT}) is a set of three independent
vertices such that there exists a path between each two of them not
passing through any vertex from the closed neighborhood of the third.
We have defined interval orders, interval, comparability, cocomparability,
and permutation graphs in the introduction. 
We will make use of the facts that $C_4$-free cocomparability graphs are
interval graphs \cite{gilmore1964characterization} and all cocomparability graphs are AT-free	\cite{golumbic1984tolerance}.

\section{Proof of Theorem~\ref{theorem:vpg}}

\label{sectionProof}

The necessity of the two conditions in Theorem~\ref{theorem:vpg} is
relatively easier to establish, and hence we do that first.
A $C_4$ has a unique B$_0$-VPG representation as shown in
Figure~\ref{fig:C4andDiamond}(a) \cite{asinowski2012vertex}.
Notice that no two vertices of an	induced $C_4$ can be represented by
collinear paths in a \B0 representation.
In contrast, one can see that in any \B0 representation of a $C_3$, at
least two of its three vertices have to be represented by collinear paths.  
Moreover, in any B$_0$-VPG representation of a diamond, the endpoints of
the diamond diagonal have to be represented by collinear paths as shown in
Figure \ref{fig:C4andDiamond}(b)
\cite{golumbic2013intersectionForDiamond}.  
Since collinearity is transitive, any two vertices which are diamond
related have to be represented by collinear paths.  
This shows the necessity of the the first condition in Theorem
\ref{theorem:vpg}. 
Notice that $\overline{C_6}$ has a $C_3$ in which each pair of vertices is
part of an induced $C_4$.  
Being part of $C_3$ forces two of the corresponding three paths to be
collinear which prevents a B$_0$-VPG representation of the corresponding
induced $C_4$. 
Hence the necessity of the second condition.  
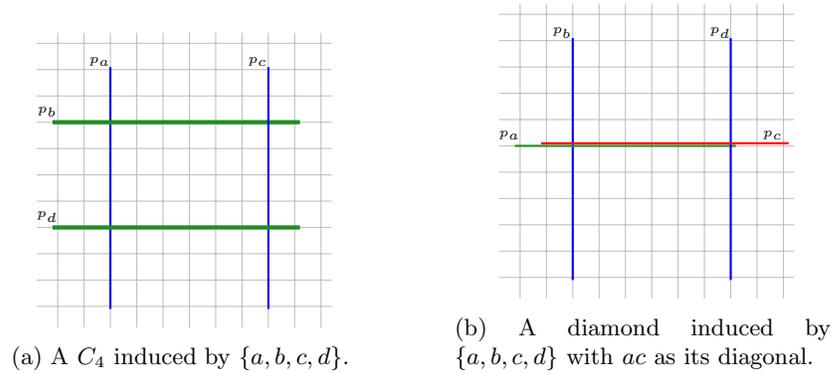
\begin{figure}[h]
	\centering
	\begin{subfigure}[b]{0.58\linewidth}
		\centering
		\begin{tikzpicture}[scale=.7]		
		
		\draw[step=0.5,lightgray,thin] (0.1,0.1) grid (5.7,5.7);
		
		\node (p1) at (1.3,5.15) {\tiny $p_a$};
		\draw [blue, thick, shorten <= -2pt, shorten >=-2pt] (1.5,0.55) -- (1.5,4.95);					
		
		\node (p1) at (0.3,4.2) {\tiny $p_b$};
		\draw [forestgreen, ultra thick, shorten <= -2pt, shorten >=-2pt] (0.5,4) -- (5,4);
		
		\node (p1) at (4.3,5.15) {\tiny $p_c$};
		\draw [blue, thick, shorten <= -2pt, shorten >=-2pt] (4.5,0.55) -- (4.5,4.95);
		
		\node (p1) at (0.3,2.2) {\tiny $p_d$};
		\draw [forestgreen, ultra thick, shorten <= -2pt, shorten >=-2pt] (0.5,2) -- (5,2);			
		
		\end{tikzpicture}
		\caption{A $C_4$ induced by $\{a,b,c,d\}$.}
	\end{subfigure}	
	\begin{subfigure}[b]{0.41\linewidth}
		\centering
		\begin{tikzpicture}[scale=.7]		
		
		\draw[step=0.5,lightgray,thin] (0.1,0.1) grid (5.7,5.7);
		
		\node (p1) at (1.3,5.15) {\tiny $p_b$};
		\draw [blue, thick, shorten <= -2pt, shorten >=-2pt] (1.5,0.55) -- (1.5,4.95);					
		
		\node (p1) at (0.3,3.2) {\tiny $p_a$};
		\draw [forestgreen, thick, shorten <= -2pt, shorten >=-2pt] (0.5,3) -- (4.5,3);
		
		\node (p1) at (4.3,5.15) {\tiny $p_d$};
		\draw [blue, thick, shorten <= -2pt, shorten >=-2pt] (4.5,0.55) -- (4.5,4.95);
		
		\node (p1) at (5.3,3.2) {\tiny $p_c$};
		\draw [red, thick, shorten <= -2pt, shorten >=-2pt] (1,3.05) -- (5.5,3.05);			
		
		\end{tikzpicture}
		\caption{A diamond induced by $\{a,b,c,d\}$ with $ac$ as its diagonal.}
	\end{subfigure}
	\caption{The unique B$_0$-VPG representation of $C_4$ and a diamond.}
	\label{fig:C4andDiamond}
\end{figure}

The three step algorithm (Algorithm \ref{algo:\B0drawing}) and the proof of its correctness (Appendix \ref{appAlgo}) are devoted to showing that these two necessary
conditions are sufficient to construct a \B0 representation of a
cocomparability graph.
The construction is completed in three steps.
We start with a cocomparability graph $G$ satisfying the two conditions of
Theorem \ref{theorem:vpg}. 
In the first step, we contract a subset of edges of $G$ to obtain a
bipartite minor $R_G$ of $G$ with a couple of additional properties.
A set of vertices in G which gets represented by a single vertex in $R_G$ after all the edge contractions is referred to as a \emph{branch set} of $R_G$. We will denote the vertices in $R_G$ by the corresponding branch sets.
In the second step, for each of the subgraphs of $G$ induced by each branch
set of $R_G$, we find an interval representation, again with a few
additional properties.
In the third and final step, we fit all the above interval representations
together to get a \B0 representation of $G$.

Before proceeding to the algorithm, we state in the next section some of the known results which ease our construction.

\subsection{Preliminaries}

An ordering $\sigma$ of $V$ of a graph $G(V,E)$ is called a
\emph{cocomparability ordering} or an \emph{umbrella-free ordering} if for all
three vertices in $x <_\sigma y <_\sigma z$, adjacency of $x$ and $z$
implies that at least one of the other pairs are adjacent. If not,
$(x,y,z)$ is called an \emph{umbrella} in $\sigma$. 

\begin{lemma}[\cite{kratsch1993domination}]
	\label{lemma:cocompImplies}
	A graph $G$ is a cocomparability graph if and only if there is a cocomparability ordering $\sigma$ of the vertices of $G$.
\end{lemma}
%

\begin{definition}
	Given a graph $G$ and a total ordering $\sigma$ of $V(G)$, a triple $(u,v,w)$ of vertices of $G$ where $u \prec_\sigma v \prec_\sigma w$ is called a \emph{forbidden triple} if there exists a path from $u$ to $w$ without containing a vertex from the closed neighborhood of $v$.
\end{definition}

\begin{lemma}
	\label{lemma:forbiddentriple_contradiction}
	Any umbrella free ordering is forbidden triple free.
\end{lemma}
\begin{proof}
	Let $\sigma$ be an umbrella free ordering.
	Assume a forbidden triple $u \prec_\sigma v \prec_\sigma w$ exists. 
	Thus there exists a path from $u$ to $w$ without containing a vertex from the closed neighborhood of $v$. If we arrange the vertices of the path together with $v$ in an order respecting $\sigma$, there exists two adjacent vertices  $u_1$ and $w_1$ among them such that $ u_1 \prec_\sigma v \prec_\sigma w_1$. Thus $(u_1, v, w_1)$ forms an umbrella in $\sigma$ which is a contradiction.
\end{proof}

	\begin{lemma}[\cite{chejnovska2015optimisation}]
		\label{lemma:cocompEdgeContraction}
		Cocomparability is preserved under edge contraction.
	\end{lemma}

A $2+2$ is a poset containing four elements where every element is comparable with exactly one element.
A $2+2$ poset corresponds to an induced $C_4$ in the complement of its comparability graph. 
Thus an interval order cannot contain a $2+2$. Similarly, if a poset does not have 
a $2+2$, then the complement of all of its comparability graphs will be $C_4$-free cocomparability graphs.

\begin{theorem}[Fishburn-1970] \emph{[\cite{fishburn1970intransitive}, Theorem 6.29 in \cite{trotter2016applied}]}
	\label{theorem:fishburn}
	A poset is an interval order if and only if it excludes $2+2$.
\end{theorem} 

Consider a bipartite graph $G(A\cup B, E)$. An ordering $\sigma$ of $A$ is said to have \emph{adjacency property} if the neighborhood of every vertex of $B$ is consecutive in $\sigma$. 
Here $G$ is called \emph{convex} if there exists an ordering $\sigma$ of $A$ with the adjacency property and \emph{biconvex} if it is convex and there exists an ordering $\tau$ of $B$ with the adjacency property.
Bipartite permutation graphs are biconvex graphs \cite{spinrad1987bipartite}.
\begin{theorem}[\cite{cohen2016posets}] \label{thm:bipartiteVPG}
	Bipartite permutation graphs are \B0.
\end{theorem}
\subsection{\B0 Algorithm} \label{subsectAlgo}

	We see a three-step algorithm to construct a \B0 representation for any arbitrary cocomparability graph satisfying the conditions of Theorem \ref{theorem:vpg}. 
	Figure \ref{fig:vpgDrawing} helps to understand the algorithm easily.
	The first step is depicted in Figure \ref{vpgDD}, \ref{vpgReducedDD}, \ref{vpgRelabeling} and
	the final drawing in the third step is shown in Figure \ref{vpgD}.	
		
	In the following definition, we assume that any self loop produced by an
edge contraction is removed and any parallel edges formed by an edge
contraction is represented by a single edge in the minor.

\begin{definition}[dd-minor]
	A \emph{dd-minor} of graph $G$ is the graph obtained by contracting every
	diamond diagonal in $G$. 
\end{definition}

\begin{definition}[Reduced dd-minor]
	A \emph{reduced dd-minor} of graph $G$ is a minimal graph $R_G$ that can be
	obtained by edge contractions of the dd-minor of $G$ such that no
	branch-set of $R_G$ contains more than one vertex of an induced $C_4$ in
	$G$.
\end{definition}
\begin{remark}
	Though the dd-minor exists for every graph, a reduced dd-minor does not
	exist for every graph. A necessary and sufficient condition for the
	existence of a reduced dd-minor for a graph $G$ is that no two vertices
	of an induced $C_4$ in $G$ should be diamond related.
\end{remark}
If an edge $xy$ ($x \prec_\sigma y$) is contracted to a new vertex, then the new vertex is placed at the position of $x$ in $\sigma$ and labeled as $x$ itself.
This results in a new order $\sigma^\prime$ which is a subsequence of $\sigma$.
By Lemma \ref{lemma:cocompEdgeContraction}, $\sigma^\prime$ is an umbrella-free ordering.
Thus after all the edge contractions to get the minimal graph $R_G$, we get an umbrella-free ordering $\sigma_{R_G}$ of $V(R_G)$ which is a subsequence of $\sigma$. 
This is sufficient to say that $R_G$ is also a cocomparability graph by Lemma \ref{lemma:cocompImplies}.
In fact, every vertex $B$ of $R_G$ is represented in $\sigma_{R_G}$ by the leftmost (under $\sigma$) vertex $b$ in the branch set $B$. 

For any two branch sets $B_i$ and $B_j$ of $R_G$, $B_{j,i}$ denotes the vertices in $B_j$ which have a neighbor in $B_i$.

\begin{lemma}[Proof in Appendix \ref{appAlgo}]
	\label{lemma:R_G}
	The following claims on the cocomparability graph $R_G$ are true.
	\begin{enumerate}
		\item \label{claimClique} For any two adjacent branch sets $B_1$ and $B_2$, the set
		$B_{1,2} \cup B_{2,1}$ induces a clique in $G$.
		\item \label{claimConsecutive} \noindent If $B_0$, $B_1$, $B_2$ form consecutive vertices of a
		$C_3$ or an induced $C_4$ in $R_G$ then $B_{1,0} \cap B_{1,2} \neq \emptyset$. Moreover, if $B_0, \ldots, B_{k-1}$ is a $C_3$ or an induced $C_4$ in $R_G$, then there exists an induced cycle $b_0, \ldots, b_{k-1}$ in $G$  where each $b_i \in B_i$.
		\item \label{claimBipartite} $R_G$ is a bipartite permutation graph.
	\end{enumerate} 
\end{lemma}

	\subsubsection{Relabeling of $V(R_G)$}\label{relabeling}
It is clear from Lemma \ref{lemma:R_G}.\ref{claimBipartite} that the reduced dd-minor $R_G$ of the cocomparability graph $G$ is a bipartite permutation graph.  
Given the umbrella-free ordering $\sigma$ of $G$, we inherited the umbrella-free ordering $\sigma_{R_G}$ for $R_G$ which respects $\sigma$.
In the algorithm, we label each branch set of the left part of $R_G$ with $B_1,B_3, \ldots$ (odd indices) such that $i<j$ implies that in the order $\sigma$, the leftmost vertex in $B_i$ is to the left of the leftmost vertex in $B_j$. Similarly we label each branch set of the right part of $R_G$ with $B_0,B_2, \ldots$ (even indices) such that $i<j$ implies that in the order $\sigma$, the leftmost vertex in $B_i$ is to the left of the leftmost vertex in $B_j$.

Henceforth, we slightly abuse the notation $\prec_\sigma$ for the branch sets of the reduced dd-minor of $G$ in the following way. For any two such branch sets $B_i$ and $B_j$, $B_i \prec_\sigma B_j$ if $\forall x \in B_i, \forall y \in B_j$, $x \prec_\sigma y$.
Thus $B_i$ and $B_j$ are said to be \emph{separated in $\sigma$} if either $B_i \prec_\sigma B_j$ or $B_j \prec_\sigma B_i$.

\begin{lemma}[Proof in Appendix \ref{appAlgo}]
	\label{lemma:branchContinuity}
	For any two branch sets $B_i,B_j$ of the same parity if $i<j$, then $B_i \prec_\sigma B_j$.
\end{lemma}

	\begin{lemma}[Proof in Appendix \ref{appAlgo}]
	\label{lemma:interval_R_G}
	For each branch set $B_i$ of $R_G$, $G[B_i^*]$ is an interval graph.
	Moreover, $G[B_i^*]$ has an interval representation $\mathcal{I}_i^*$
	satisfying	the following properties.
	\begin{enumerate}[label=(\roman*)]
		\item \label{interval1}
		For all $x, y \in B_i^*$, we have the interval for $x$ to the left
		of interval for $y$ if and only if $x \prec_P y$.
		\item \label{interval2}
		For each neighbor $B_j$ of $B_i$, all the intervals corresponding
		to the vertices in $B_{j,i}$ are point intervals at a point
		$p_{j,i}$.
		\item \label{interval3}
		If $B_j$ and $B_k$ are two neighbors of $B_i$ such that $j<k$, then
		the point $p_{j,i}$ is to the left of the point $p_{k,i}$ in
		$\mathcal{I}_i^*$. 
	\end{enumerate} 
\end{lemma}
\begin{figure}
	
	\centering
	\begin{subfigure}[b]{0.5\linewidth}
		\centering
		\begin{tikzpicture}[scale=.7]
		\node (e) at (0,0) {e};
		\node [below of=e, node distance=.6in] (g) {g};
		\node [right of=e, node distance=.8in] (f) {f};
		\node [right of=g, node distance=.8in] (h) {h};
		\node [above of=e, node distance=.4in] (c) {c};
		\node [above of=f, node distance=.4in] (d) {d};
		\node [above of=c, node distance=.4in] (a) {a};
		\node [above of=d, node distance=.4in] (b) {b};
		\node [below right of=g, node distance=.6in] (j) {j};
		\node [above of=j, node distance=.65in] (i) {i};
		
		\draw [black, thick, shorten <= -2pt, shorten >=-2pt] (a) -- (b);
		\draw [black, thick, shorten <= -2pt, shorten >=-2pt] (a) -- (c);
		\draw [black, thick, shorten <= -2pt, shorten >=-2pt] (c) -- (d);
		\draw [black, thick, shorten <= -2pt, shorten >=-2pt] (b) -- (d);
		\draw [black, thick, shorten <= -2pt, shorten >=-2pt] (c) -- (e);
		\draw [black, thick, shorten <= -2pt, shorten >=-2pt] (e) -- (f);
		\draw [black, thick, shorten <= -2pt, shorten >=-2pt] (e) -- (d);
		\draw [black, thick, shorten <= -2pt, shorten >=-2pt] (e) -- (g);
		\draw [black, thick, shorten <= -2pt, shorten >=-2pt] (g) -- (h);
		\draw [black, thick, shorten <= -2pt, shorten >=-2pt] (f) -- (h);
		\draw [black, thick, shorten <= -2pt, shorten >=-2pt] (i) -- (g);
		\draw [black, thick, shorten <= -2pt, shorten >=-2pt] (i) -- (f);
		\draw [black, thick, shorten <= -2pt, shorten >=-2pt] (i) -- (j);
		\draw [black, thick, shorten <= -2pt, shorten >=-2pt] (g) -- (j);
		\draw [black, thick, shorten <= -2pt, shorten >=-2pt] (h) -- (j);
		\draw [black, thick, shorten <= -2pt, shorten >=-2pt] (i) -- (h);
		\draw [black, thick, shorten <= -2pt, shorten >=-2pt] (d) -- (f);
		\draw [black, thick, shorten <= -2pt, shorten >=-2pt] (c) -- (f);
		\end{tikzpicture}
		\caption{The \emph{cocomparability} graph $G$.}\label{vpgG}
	\end{subfigure}
	\begin{subfigure}[b]{0.48\linewidth}
		\centering
		\begin{tikzpicture}[scale=.7]
		\node (e) at (0,0) {e};
		\node [below of=e, node distance=.6in] (g) {g};
		\node [right of=e, node distance=.8in] (f) {f};
		\node [right of=g, node distance=.8in] (h) {h};
		\node [above of=e, node distance=.4in] (c) {c};
		\node [above of=f, node distance=.4in] (d) {d};
		\node [above of=c, node distance=.4in] (a) {a};
		\node [above of=d, node distance=.4in] (b) {b};
		\node [below right of=g, node distance=.6in] (j) {j};
		\node [above of=j, node distance=.65in] (i) {i};
		
		\draw [black, thick, shorten <= -2pt, shorten >=-2pt] (a) -- (b);
		\draw [black, thick, shorten <= -2pt, shorten >=-2pt] (a) -- (c);
		\draw [black, thick, shorten <= -2pt, shorten >=-2pt] (c) -- (d);
		\draw [black, thick, shorten <= -2pt, shorten >=-2pt] (b) -- (d);
		\draw [black, thick, shorten <= -2pt, shorten >=-2pt] (c) -- (e);
		\draw [black, thick, shorten <= -2pt, shorten >=-2pt] (e) -- (f);
		\draw [black, thick, shorten <= -2pt, shorten >=-2pt] (e) -- (d);
		\draw [black, thick, shorten <= -2pt, shorten >=-2pt] (e) -- (g);
		\draw [black, thick, shorten <= -2pt, shorten >=-2pt] (g) -- (h);
		\draw [black, thick, shorten <= -2pt, shorten >=-2pt] (f) -- (h);
		\draw [black, thick, shorten <= -2pt, shorten >=-2pt] (i) -- (g);
		\draw [black, thick, shorten <= -2pt, shorten >=-2pt] (i) -- (f);
		\draw [black, thick, shorten <= -2pt, shorten >=-2pt] (i) -- (j);
		\draw [black, thick, shorten <= -2pt, shorten >=-2pt] (g) -- (j);
		\draw [black, thick, shorten <= -2pt, shorten >=-2pt] (h) -- (j);
		\draw [black, thick, shorten <= -2pt, shorten >=-2pt] (i) -- (h);
		\draw [black, thick, shorten <= -2pt, shorten >=-2pt] (d) -- (f);
		\draw [black, thick, shorten <= -2pt, shorten >=-2pt] (c) -- (f);
		
		\draw[rotate=60] (-.45,-2.8) ellipse (.6cm and 1.5cm); 
		\draw (2.85,0) circle (0.4cm); 
		\draw (0,0) circle (0.4cm); 
		\draw (0,1.5) circle (0.4cm); 
		\draw (2.85,1.5) circle (0.4cm); 
		\draw (0,2.9) circle (0.4cm); 
		\draw (2.85,2.9) circle (0.4cm); 
		\draw (0,-2.1) circle (0.4cm); 
		\draw (1.5,-3.6) circle (0.4cm); 
		\end{tikzpicture}
		\caption{The \emph{dd-minor} of $G$.}\label{vpgDD}
		
	\end{subfigure}
	
	\begin{subfigure}[b]{0.5\linewidth}
		\centering
		\begin{tikzpicture}[scale=.7]
		\node (e) at (0,0) {e};
		\node [below of=e, node distance=.6in] (g) {g};
		\node [right of=e, node distance=.8in] (f) {f};
		\node [right of=g, node distance=.8in] (h) {h};
		\node [above of=e, node distance=.4in] (c) {c};
		\node [above of=f, node distance=.4in] (d) {d};
		\node [above of=c, node distance=.4in] (a) {a};
		\node [above of=d, node distance=.4in] (b) {b};
		\node [below right of=g, node distance=.6in] (j) {j};
		\node [above of=j, node distance=.65in] (i) {i};
		
		\draw [black, thick, shorten <= -2pt, shorten >=-2pt] (a) -- (b);
		\draw [black, thick, shorten <= -2pt, shorten >=-2pt] (a) -- (c);
		\draw [black, thick, shorten <= -2pt, shorten >=-2pt] (c) -- (d);
		\draw [black, thick, shorten <= -2pt, shorten >=-2pt] (b) -- (d);
		\draw [black, thick, shorten <= -2pt, shorten >=-2pt] (c) -- (e);
		\draw [black, thick, shorten <= -2pt, shorten >=-2pt] (e) -- (f);
		\draw [black, thick, shorten <= -2pt, shorten >=-2pt] (e) -- (d);
		\draw [black, thick, shorten <= -2pt, shorten >=-2pt] (e) -- (g);
		\draw [black, thick, shorten <= -2pt, shorten >=-2pt] (g) -- (h);
		\draw [black, thick, shorten <= -2pt, shorten >=-2pt] (f) -- (h);
		\draw [black, thick, shorten <= -2pt, shorten >=-2pt] (i) -- (g);
		\draw [black, thick, shorten <= -2pt, shorten >=-2pt] (i) -- (f);
		\draw [black, thick, shorten <= -2pt, shorten >=-2pt] (i) -- (j);
		\draw [black, thick, shorten <= -2pt, shorten >=-2pt] (g) -- (j);
		\draw [black, thick, shorten <= -2pt, shorten >=-2pt] (h) -- (j);
		\draw [black, thick, shorten <= -2pt, shorten >=-2pt] (i) -- (h);
		\draw [black, thick, shorten <= -2pt, shorten >=-2pt] (d) -- (f);
		\draw [black, thick, shorten <= -2pt, shorten >=-2pt] (c) -- (f);
		
		\draw[rotate=60] (-.45,-2.8) ellipse (.6cm and 1.5cm); 
		\draw[rotate=90] (.8,0) ellipse (1.2cm and .5cm); 
		\draw[rotate=90] (.8,-2.9) ellipse (1.2cm and .5cm); 
		\draw (0,2.9) circle (0.4cm); 
		\draw (2.85,2.9) circle (0.4cm); 
		\draw[rotate=50] (-1.8,-2.4) ellipse (.6cm and 1.5cm); 
		\end{tikzpicture}
		\caption{A \emph{reduced dd-minor} $R_G$ of $G$ \\(also bipartite \emph{permutation} graph).}\label{vpgReducedDD}
	\end{subfigure}
	\begin{subfigure}[b]{0.48\linewidth}
		
		\centering
		\begin{tikzpicture}[scale=.7]
		
		\node (a) at (0,0) {a};
		\node [right of=a, node distance=.7in] (b) {b};
		\node [above of=a, node distance=.5in] (d) {d};
		\node [above of=b, node distance=.5in] (c) {c};
		\node [left of=d, node distance=.5in] (e) {e};
		\node [right of=c, node distance=.5in] (f) {f};
		
		\node [above of=d, node distance=.5in] (h) {h};
		\node [above of=c, node distance=.5in] (j) {j};
		\node [left of=h, node distance=.6in] (i) {i};
		\node [right of=j, node distance=.6in] (g) {g};
		
		\draw [black, thick, shorten <= -2pt, shorten >=-2pt] (a) -- (e);
		\draw [black, thick, shorten <= -2pt, shorten >=-2pt] (a) -- (d);
		\draw [black, thick, shorten <= -2pt, shorten >=-2pt] (a) -- (f);
		\draw [black, thick, shorten <= -2pt, shorten >=-2pt] (b) -- (e);
		\draw [black, thick, shorten <= -2pt, shorten >=-2pt] (b) -- (c);
		\draw [black, thick, shorten <= -2pt, shorten >=-2pt] (b) -- (f);
		
		\draw [black, thick, shorten <= -2pt, shorten >=-2pt] (e) -- (i);
		\draw [black, thick, shorten <= -2pt, shorten >=-2pt] (e) -- (h);
		\draw [black, thick, shorten <= -2pt, shorten >=-2pt] (e) -- (j);
		
		\draw [black, thick, shorten <= -2pt, shorten >=-2pt] (d) -- (i);
		\draw [black, thick, shorten <= -2pt, shorten >=-2pt] (d) -- (h);
		\draw [black, thick, shorten <= -2pt, shorten >=-2pt] (d) -- (j);
		\draw [black, thick, shorten <= -2pt, shorten >=-2pt] (d) -- (g);
		
		\draw [black, thick, shorten <= -2pt, shorten >=-2pt] (c) -- (i);
		\draw [black, thick, shorten <= -2pt, shorten >=-2pt] (c) -- (h);
		\draw [black, thick, shorten <= -2pt, shorten >=-2pt] (c) -- (j);
		\draw [black, thick, shorten <= -2pt, shorten >=-2pt] (c) -- (g);
		
		\draw [black, thick, shorten <= -2pt, shorten >=-2pt] (f) -- (j);
		\draw [black, thick, shorten <= -2pt, shorten >=-2pt] (f) -- (g);
		\end{tikzpicture}
		\caption{Hasse diagram of an arbitrary poset $P_G(V,\prec_P)$ of $G$. Choose and fix a \emph{linear extension} $\sigma = (b,a,c,d,f,e,g,i,h,j)$.}\label{vpgPoset}
	\end{subfigure}
	
	\begin{subfigure}[b]{0.47\linewidth}
		\centering
		\begin{tikzpicture}[scale=.7]
		
		\node (b) at (0,0) {b};
		\node [right of=b, node distance=1.5in] (a) {a};	
		\node [below of=b, node distance=.6in] (c) {c};
		\node [below of=c, node distance=.3in] (e) {e};
		
		\node [below of=e, node distance=.6in] (i) {i};
		\node [below of=i, node distance=.3in] (h) {h};
		
		\node [right of=c, node distance=1.5in] (d) {d};
		\node [below of=d, node distance=.3in] (f) {f};
		
		\node [right of=i, node distance=1.5in] (g) {g};
		\node [below of=g, node distance=.3in] (j) {j};

		\draw [black, thick, shorten <= -2pt, shorten >=-2pt] (a) -- (b);
		\draw [black, thick, shorten <= -2pt, shorten >=-2pt] (a) -- (c);
		\draw [black, thick, shorten <= -2pt, shorten >=-2pt] (c) -- (d);
		\draw [black, thick, shorten <= -2pt, shorten >=-2pt] (b) -- (d);
		\draw [black, thick, shorten <= -2pt, shorten >=-2pt] (c) -- (e);
		\draw [black, thick, shorten <= -2pt, shorten >=-2pt] (e) -- (f);
		\draw [black, thick, shorten <= -2pt, shorten >=-2pt] (e) -- (d);
		\draw [black, thick, shorten <= -2pt, shorten >=-2pt] (e) -- (g);
		\draw [black, thick, shorten <= -2pt, shorten >=-2pt] (g) -- (h);
		\draw [black, thick, shorten <= -2pt, shorten >=-2pt] (f) -- (h);
		\draw [black, thick, shorten <= -2pt, shorten >=-2pt] (i) -- (g);
		\draw [black, thick, shorten <= -2pt, shorten >=-2pt] (i) -- (f);
		\draw [black, thick, shorten <= -2pt, shorten >=-2pt] (i) -- (j);
		\draw [black, thick, shorten <= -2pt, shorten >=-2pt] (g) -- (j);
		\draw [black, thick, shorten <= -2pt, shorten >=-2pt] (h) -- (j);
		\draw [black, thick, shorten <= -2pt, shorten >=-2pt] (i) -- (h);
		\draw [black, thick, shorten <= -2pt, shorten >=-2pt] (d) -- (f);
		\draw [black, thick, shorten <= -2pt, shorten >=-2pt] (c) -- (f);
		
		\draw (0,0) circle (0.4cm);
		\node (b0) at (-.5, -.5) {$B_1$};
		
		\draw (0,-2.7) circle (.8cm);
		\node (b1) at (-.5,-3.7) {$B_3$};
		
		\draw (0,-6) circle (.8cm);
		\node (b3) at (-.5,-6.9) {$B_5$};
		
		\draw (5.5,0) circle (0.4cm);
		\node (b0) at (6.1,-.5) {$B_0$};
		
		\draw (5.5,-2.7) circle (0.8cm);
		\node (b2) at (6.1,-3.7) {$B_2$};
		
		\draw (5.5,-6) circle (.8cm);
		\node (b4) at (6.1,-6.9) {$B_4$};
		\end{tikzpicture}
		\caption{Labeled $R_G$ respecting $\sigma$.}\label{vpgRelabeling}
	\end{subfigure}
	\begin{subfigure}[b]{0.52\linewidth}
		\centering
		\begin{tikzpicture}[scale=.7]

		\draw[step=0.8,lightgray,thin] (0.1,0.1) grid (7.7,6.7);
		
		\node (p0v) at (1.6,0.1) {\tiny $0$};
		\node (p1v) at (2.4,0.1) {\tiny $1$};
		\node (p2v) at (3.2,0.1) {\tiny $2$};
		\node (p3v) at (4,0.1) {\tiny $3$};
		\node (p4v) at (4.8,0.1) {\tiny $4$};
		\node (p5v) at (5.6,0.1) {\tiny $5$};
		
		\node (p1v) at (2.4,6.6) {$\mathcal{I}_1$};
		\node (p3v) at (4,6.6) {$\mathcal{I}_3$};
		\node (p5v) at (5.6,6.6) {$\mathcal{I}_5$};
		
		\node (p0h) at (0,1.6) {\tiny $0$};
		\node (p1h) at (0,2.4) {\tiny $1$};
		\node (p2h) at (0,3.2) {\tiny $2$};
		\node (p0h) at (0,4) {\tiny $3$};
		\node (p4h) at (0,4.8) {\tiny $4$};
		\node (p5h) at (0,5.6) {\tiny $5$};
		
		\node (p0h) at (7.8,1.6) {$\mathcal{I}_0$};
		\node (p2h) at (7.8,3.2) {$\mathcal{I}_2$};
		\node (p4h) at (7.8,4.8) {$\mathcal{I}_4$};
		
		\draw [darkblue, thick, shorten <= -2pt, shorten >=-2pt] (2.4,1.2) -- (2.4,3.6);
		\node (pe) at (2.6,1) {\tiny $I_b$};
		\draw [darkblue, thick, shorten <= -2pt, shorten >=-2pt] (4,1.2) -- (4,3.6);
		\node (pe) at (4.2,1) {\tiny $I_c$};
		\draw [darkblue, thick, shorten <= -2pt, shorten >=-2pt] (4.1,2.8) -- (4.1,5.2);
		\node (pg) at (4.3,5.2) {\tiny $I_e$};
		\draw [darkblue, thick, shorten <= -2pt, shorten >=-2pt] (5.6,2.8) -- (5.6,5.2);
		\node (pd) at (5.4,2.8) {\tiny $I_i$};
		\draw [darkblue, thick, shorten <= -2pt, shorten >=-2pt] (5.7,2.8) -- (5.7,5.2);
		\node (pd) at (5.92,2.8) {\tiny $I_h$};
		\draw [forestgreen, thick, shorten <= -2pt, shorten >=-2pt] (2,1.6) -- (4.4,1.6);
		\node (pd) at (1.8,1.8) {\tiny $I_a$};
		\draw [forestgreen, thick, shorten <= -2pt, shorten >=-2pt] (2,3.2) -- (4.4,3.2);
		\node (pf) at (1.8,3) {\tiny $I_d$};
		\draw [forestgreen, thick, shorten <= -2pt, shorten >=-2pt] (3.6,3.3) -- (6,3.3);
		\node (pd) at (6.1,3.5) {\tiny $I_f$};
		\draw [forestgreen, thick, shorten <= -2pt, shorten >=-2pt] (3.6,4.8) -- (6,4.8);
		\node (pb) at (6.1,4.6) {\tiny $I_g$};
		\draw [forestgreen, thick, shorten <= -2pt, shorten >=-2pt] (5.2,4.9) -- (6,4.9);
		\node (pb) at (5.1,5.1) {\tiny $I_j$};
		
		%
		
		\end{tikzpicture}
		\caption{The drawing $D$ which is also a 2\--D visualization of $P_G$ by Theorem \ref{theorem:2D}.}\label{vpgD}
	\end{subfigure}
	\caption{Drawing a \B0 representation $D$ of a cocomparability graph $G$ satisfying the conditions of Theorem \ref{theorem:vpg}. Note that the collinear intersecting line segments in $D$ are drawn a little apart in order to distinguish them easily.}\label{fig:vpgDrawing}
\end{figure}
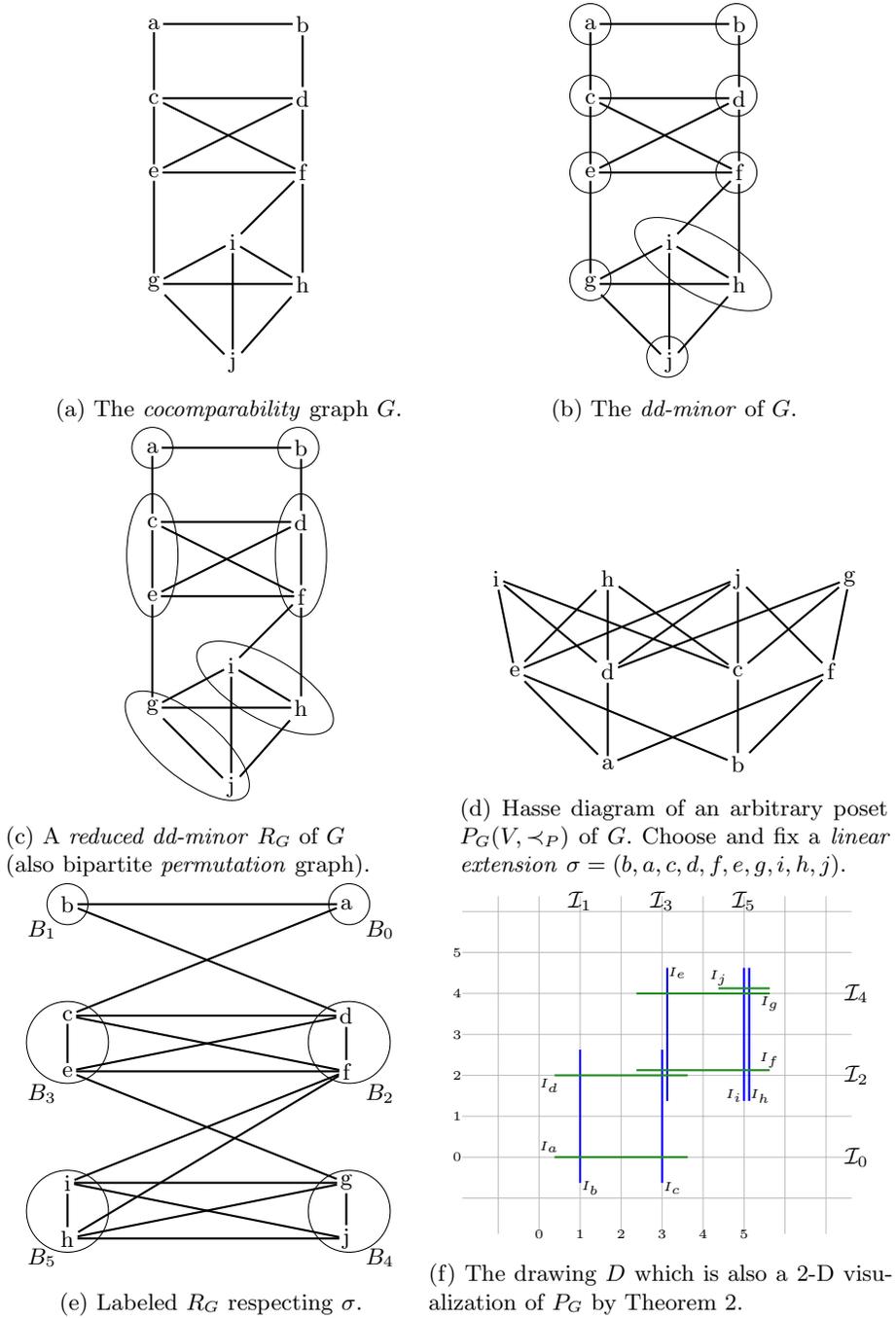
\begin{algorithm}[H]
	
	\caption{\\ \textbf{Input.} $G$ is an arbitrary but fixed  cocomparability graph satisfying the two conditions of Theorem \ref{theorem:vpg}. 
	\\ \textbf{Output.} A \B0 representation $D$ of $G$ .
	\\ \textbf{Assumptions.} \\ 1) $P_G (V(G), \prec_P)$ is an arbitrary but fixed partial order whose comparability graph is $\comp{G}$. \\ 2) $\sigma$ is an arbitrary but fixed linear extension of $P_G$ and hence an umbrella-free ordering for $G$.}
	\label{algo:\B0drawing}
	\begin{enumerate}		
		\itemsep0em
		\item Step 1: Choose an arbitrary but fixed reduced dd-minor $R_G$ of $G$ and label $V(R_G)$ as described in the above mentioned relabeling procedure.
		\item Step 2:
		\begin{enumerate}[label=(\roman*)]
			\item For each branch set $B_i$ of $R_G$, let $B_i^* = B_i \cup \{B_{j,i} : B_iB_j \in E(R_G)\}$ and we obtain an interval representation $\mathcal{I}_i^*$ of $B_i^*$ using Lemma \ref{lemma:interval_R_G}.
			\item Remove intervals of vertices in  $B_i^* \setminus B_i$ from $\mathcal{I}_i^*$ to get $\mathcal{I}_i$.
		\end{enumerate}
		\item Step 3: Construction of the drawing D using the following steps. Let $e$ and $o$ respectively (with further subscripts if needed) denote the
		even and odd indices of the branch sets of $R_G$.
			\begin{enumerate}[label=(\roman*)]
			\item 
			For each odd-indexed branch set $B_o$, $\mathcal{I}_o$ is drawn
			vertically from the point $(o,(e_1-0.5))$ to $(o,(e_2+0.5))$ where
			$B_{e_1}$ and $B_{e_2}$ are the leftmost and rightmost neighbors of
			$B_o$ in $\sigma_{R_G}$.
			\item 
			Stretch or shrink the intervals in each vertical interval
			representation $\mathcal{I}_o$ without changing their intersection
			pattern, so that for each neighbor $B_e$ of $B_o$, the point $p_{e,o}$
			is at $(o,e)$.
			This can be done since the intersection pattern of an interval
			representation with $n$ intervals is solely determined by the order of
			the corresponding $2n$ endpoints.
			
			\item 
			For each even-indexed branch set $B_{e}$, $\mathcal{I}_e$ is drawn
			horizontally from the point $((o_1-0.5),e)$ to $((o_2+0.5),e)$ where
			$B_{o_1}$ and $B_{o_2}$ are the leftmost and rightmost neighbors of
			$B_e$  in $\sigma_{R_G}$.
			\item 
			Stretch or shrink the intervals in each horizontal interval
			representation $\mathcal{I}_{e}$ without changing their intersection
			pattern, so that for each neighbor $B_o$ of $B_e$, the point $p_{o,e}$
			is at $(o,e)$.
		\end{enumerate}
	\end{enumerate}
\end{algorithm}

\begin{proposition}
	\label{proposeD}
	The B$_0$-VPG representation $D$ is precisely a B$_0$-VPG
	representation of the cocomparability graph $G$.
\end{proposition}
The proof of the above proposition is written in Appendix \ref{appAlgo}.
One can easily verify that Algorithm \ref{algo:\B0drawing} runs in polynomial time.

\section{Proof of Theorem \ref{theorem:2D}}


In this section, we fix $P(V, \prec_P)$ as the given input poset and $G$ as its
cocomparability graph. 	Let $D$ be the \B0 representation of $G$ obtained by
the construction employed in the proof of Theorem \ref{theorem:vpg}, where the
partial order $P_G$ assumed in Algorithm \ref{algo:\B0drawing}
is $P$. We argue below that for any two vertices $x$ and $y$ in $V(G)$, $x
\prec_D y$ (cf. Definition~\ref{defnDiagramOrder}) if and only if $x \prec_P y$
and thus establish Theorem~\ref{theorem:2D}.  First, we show that if $x \prec_P
y$, then $x \prec_D y$. But a simple exchange of variables is not enough to
prove the converse because $\prec_D$ is not antisymmetric in the set of all
vertical and horizontal line segments. Hence in order to complete the proof we
show that the relation $\prec_D$ is antisymmetric when restricted to the line
segments in $D$.

\subsection{If $x \prec_P y$ then $x \prec_D y$}

Recall that $\sigma$ is a linear extension of $P_G$.  Thus if $x \prec_P y$,
then $x \prec_\sigma y$ and $x$ is non-adjacent to $y$ in $G$. 

If $x$ and $y$ are in the same branch set $B_i$ of $R_G$ then clearly $I_x
\prec_{\mathcal{I}_i} I_y$ by Lemma \ref{lemma:interval_R_G}.\ref{interval1}.
Here if $I_x$ is horizontal, clearly $x$ \p{v}{h} $y$. Otherwise, $x \p{h}{v}
y$. If $x$ and $y$ are in different branch sets, $B_i$ and $B_j$ respectively,
of the same parity, then $i < j$ (Lemma \ref{lemma:branchContinuity}). Thus in
$D$, if both are of odd parity, $B_i$ is drawn to the left of $B_j$ and if both
are of even parity, $B_i$ is drawn to the bottom of $B_j$. Thus if $I_x$ is
vertical, then $x$ \p{h}{v} $y$ and if $I_x$ is horizontal, then $x$ \p{v}{h}
$y$.

If the parity is opposite, we have two sub-cases; that is $B_i$ and $B_j$ are
either non-adjacent or adjacent. If non-adjacent, the branch set $B_j$ cannot
have a neighbor $B_h$ where $h < i$. Suppose there exists such a neighbor $B_h$
for $B_j$ such that $h < i$. Clearly since $B_h$ is adjacent to $B_j$, $h$ and
$i$ are of the same parity and hence $B_h \prec_{\sigma} B_i$ (Lemma
\ref{lemma:branchContinuity}). Since $B_h$ is adjacent to $B_j$, there exists a
path from a vertex $z \in B_h$ to $y \in B_j$ in $G[B_h \cup B_j]$.  Moreover
since $B_i$ is disjoint from $B_h$ and $B_j$, $x$ has no neighbor in $B_h \cup
B_j$, the hence triple $(z,x,y)$ is a forbidden triple in $\sigma$ which is a
contradiction as per Lemma \ref{lemma:forbiddentriple_contradiction}.  Since
$B_j$ has no opposite parity neighbor $B_h$ for any $h \leq i$, the following
property can easily be verified from our drawing.  Thus if $B_i$ is of even
parity, then $B_i$ is to the bottom of $B_j$ in $D$. Otherwise, $B_i$ to the
left of $B_j$ in $D$. Hence clearly if $I_x$ is horizontal, then $I_x$ is to
the bottom of $I_y$, that is $x$ \p{v}{h} $y$. Similarly, if $I_x$ is vertical,
then $I_x$ is to the left of $I_y$, that is $x$ \p{h}{v} $y$.

Now the remaining sub-case is that $B_i$ and $B_j$ are adjacent.  The following
observation is frequently used in the remaining part of the proof.

\begin{observation} \label{obs:2Dforward}
	If the interval representations $\mathcal{I}_i$ and $\mathcal{I}_j$
	intersects, there exist at least two intersecting intervals $I_{u_1} \in
	\mathcal{I}_i$ and $I_{v_1} \in \mathcal{I}_j$. For any interval $I_u \in
	\mathcal{I}_i$, if $I_u \prec_{\mathcal{I}_i} p_{j,i}$, then $u
	\prec_\sigma v_1$ and if $p_{j,i} \prec_{\mathcal{I}_i}  I_u $ then $v_1
	\prec_\sigma u$. This is easily inferred from $\mathcal{I}_i^*$.  We can
	symmetrically argue the same for $I_v$.
\end{observation}

The intervals $I_x$ and $I_y$ do not intersect since $x$ and $y$ are
non-adjacent in $G$. If $I_x$ contains $p_{j,i}$, then $p_{i,j}$ (geometrically
coinciding with $p_{j,i}$) has to precede $I_y$ in $\mathcal{I}_j$. Otherwise
due to Observation \ref{obs:2Dforward}, we get $y \prec_\sigma x$ which is a
contradiction. Similarly if $I_y$ contains $p_{i,j}$, then $I_x$ has to
precede $p_{j,i}$ in $\mathcal{I}_i$. In both these case, if $I_x$ is
horizontal, then $x$ \p{v}{h} $y$ and if $I_x$ is vertical, then $x$ \p{h}{v}
$y$. Henceforth we assume that neither $I_x$ nor $I_y$ contains the
intersection point of $\mathcal{I}_i$ and $\mathcal{I}_j$. In this case, it is
easy to see that $x$ has no neighbors in $B_j$ and $y$ has no neighbors in
$B_i$.

In order to rule out the following scenarios, we show the existence of a
forbidden triple in $\sigma$ which leads to a contradiction as per the Lemma
\ref{lemma:forbiddentriple_contradiction}.

If $I_y \prec_{\mathcal{I}_j} p_{i,j}$, then $x \prec_\sigma y \prec_\sigma
u_1$ as per Observation \ref{obs:2Dforward}. Thus $(x,y,u_1)$ is a forbidden
triple.

If $p_{j,i} \prec_{\mathcal{I}_i}  I_x$, then $v_1 \prec_\sigma x \prec_\sigma
y$ as per Observation \ref{obs:2Dforward}. Thus $(v_1,x,y)$ is a forbidden
triple.

Hence $p_{i,j} \prec_{\mathcal{I}_j} I_y$ and $I_x \prec_{\mathcal{I}_i} p_{j,i}$. In this case, if $I_x$ is horizontal, $x$ \p{v}{h} $y$ and if $I_x$ is vertical, we get $x$ \p{h}{v} $y$. 
Moreover, in both these cases, $I_x$ is to the left and to the bottom of $I_y$.

Thus we have proved that when $x \prec_P y$, we get that $x$ \p{v}{h} $y$ when $I_x$ is horizontal or $x$ \p{h}{v} $y$ when $I_x$ is vertical in the \B0 representation $D$ of $G$. That is $x \prec_D y$.
If $x$ and $y$ are incomparable in $P_G$, then they are adjacent in $G$ and the corresponding intervals intersect in $D$. 

\subsection{Antisymmetry of  $\prec_D$}	

\begin{observation} \label{obs:2Dconverse1}
	If two opposite parity branch sets are non-intersecting, then one of them is entirely to the bottom left of the other in $D$. Hence for all $I_x$ in the bottom left branch set and for all $I_y$ in the top right branch set, $x \prec_D y$.
\end{observation}

\noindent \textit{Justification.}
When two branch sets are non-intersecting, they are \emph{separated in
	$\sigma$} since otherwise, there will exist a forbidden triple $u \prec_\sigma
v \prec_\sigma w$ such that $u$ and $w$ are in the one branch set and $v$ in
the other branch set.  Without loss of generality, let $B_i \prec_\sigma B_j$.
We claim that $B_i$ is entirely to the bottom left of $B_j$. Assume not.
That is either $B_i$ has an opposite parity neighbor $B_k$ for some $k > j$ or
$B_j$ has an opposite parity neighbor $B_h$ for some $h < i$ or both. 
In the first case, there exists a forbidden triple $(x, y, z)$ for any $x\in B_i$, $y \in B_j$ and $z \in B_k$ 
which is a contradiction by Lemma \ref{lemma:forbiddentriple_contradiction}.
Similarly in the second case, there exists a forbidden triple $(w,x,y)$ for any
$w \in B_h$, $x\in B_i$ and $y \in B_j$ which is again a contradiction by Lemma
\ref{lemma:forbiddentriple_contradiction}.

\begin{observation}\label{obs:2Dconverse2}
	In $D$, there is no line segment $I_b$ which is to the bottom right of a line segment $I_t$ of an opposite parity branch set.
\end{observation}
\noindent \textit{Justification.} Assume $I_t$ is in $\mathcal{I}_i$ and $I_b$ is in $\mathcal{I}_j$. 
The branch sets $B_i$ and $B_j$ are of the opposite parity. If they are non-adjacent, then by Observation \ref{obs:2Dconverse1}, either $I_t$ has to be bottom left of $I_b$ or $I_b$ has to be bottom left of $I_t$. 
In both cases, $I_b$ can not be bottom right of $I_t$. Now we consider the case when the branch sets are adjacent. 
Since $I_t$ and $I_b$ are non-intersecting, there exists intersecting intervals $I_{t_1} \in \mathcal{I}_i$ and $I_{b_1} \in \mathcal{I}_j$ as per Observation \ref{obs:2Dforward}.
In $\sigma$, either $t$ precedes $b$ or $b$ precedes $t$. These result in the forbidden triple either $(t,b,t_1)$ or $(b,t,b_1)$ respectively leading to a contradiction by Lemma \ref{lemma:forbiddentriple_contradiction}.

\begin{lemma} \label{lemma:2Dconverse3}
	Any two non-intersecting line segments $I_x$ and $I_y$ in $D$ satisfy
	either $x \prec_D y$ or $y \prec_D x$, but not both. In particular, the
	relation $\prec_D$ is antisymmetric.
\end{lemma}

\begin{proof}
	Since $D$ is a \B0 representation of $G$, $x$ and $y$ are nonadjacent in
	$G$ and hence comparable in $P$.  That is, either $x \prec_P y$ or $y
	\prec_P x$. Therefore, $x \prec_D y$ or $y \prec_D x$. Hence it is enough
	to verify that $x \prec_D y$ and $y \prec_D x$ cannot both be true.  This is
	easily verified when $I_x$ and $I_y$ are both horizontal or both vertical.
	Hence we can assume without loss of generality that $I_x$ is horizontal and
	$I_y$ is vertical. If both $x \prec_D y$ and $y \prec_D x$, then $I_x$ has
	to be to the bottom right of $I_y$. This contradicts
	Observation~\ref{obs:2Dconverse2}. 
\end{proof}

This concludes the proof of Theorem~\ref{theorem:2D}.
For every two incomparable elements in $\prec_P$, the corresponding line
segments intersect in $D$. For every two comparable elements $x,y \in V(G)$, if
$x \prec_P y$, then $x \prec_D y$. By exchange of variables, if $y \prec_P x$,
then $y \prec_D x$. Lemma~\ref{lemma:2Dconverse3} asserts that exactly one
of the above is true for any two non-intersecting line segments. Hence $x
\prec_D y$ only when $x \prec_P y$.  Thus the relation $\prec_D$ is isomorphic
to the relation $\prec_P$.

\newpage
\begin{subappendices}
	\renewcommand{\thesection}{\Alph{section}}%

	\section{Proof of Correctness of Algorithm \ref{algo:\B0drawing}} \label{appAlgo}
	
	\subsection{Step 1. Reduced dd-minor $R_G$ of $G$}
	In this step, we contract a subset of edges of $G$
	to obtain a bipartite minor $R_G$ of $G$ which satisfies the additional
	properties listed in Lemma~\ref{lemma:R_G}. 
	We also label the branch-sets of $R_G$ so that the order implied by this
	labeling has the adjacency property (Lemma~\ref{lemma:convexR_G}).

	We see a few additional properties of $R_G$ using which we draw the B$_0$-VPG representation of $G$. These properties are clubbed into the following lemma.
	
		
	\bigskip \noindent \textbf{Lemma \ref{lemma:R_G}} Statement:
		\emph{The following claims on the cocomparability graph $R_G$ are true.
		\begin{enumerate}
			\item For any two adjacent branch sets $B_1$ and $B_2$, the set
			$B_{1,2} \cup B_{2,1}$ induces a clique in $G$.
			\item \noindent If $B_0$, $B_1$, $B_2$ form consecutive vertices of a
			$C_3$ or an induced $C_4$ in $R_G$ then $B_{1,0} \cap B_{1,2} \neq \emptyset$. Moreover, if $B_0, \ldots, B_{k-1}$ is a $C_3$ or an induced $C_4$ in $R_G$, then there exists an induced cycle $b_0, \ldots, b_{k-1}$ in $G$  where each $b_i \in B_i$.
			\item $R_G$ is a bipartite permutation graph.
		\end{enumerate}}
	\begin{proof}
		We prove the claims in the order they are listed.
		
		\bigskip \label{claimCliqueproof} \noindent \textbf{Proof of Claim 1}: First we need to prove that for any two adjacent branch sets $B_1$ and $B_2$, the subgraph of $G$ induced by $B_1 \cup B_2$ is an interval graph. 		
		Since no branch set of $R_G$ contains more than one vertex of
		an induced $C_4$, $G[B_1 \cup B_2]$ is a $C_4$-free cocomparability graph 
		and therefore an interval graph \cite{gilmore1964characterization}.
		
		Now consider any two non-adjacent vertices $x,y \in B_{1,2}$ having a common neighbor $z_1 \in B_1$. 
		Assume $x$ and $y$ are adjacent to a single vertex $z_2 \in B_{2,1}$. 
		In order to avoid $x$ and $y$ being part of an induced $C_4$, the possible edge is $z_1z_2$.
		Clearly $z_1z_2$ is a diamond diagonal edge across two branch sets. This is a contradiction. 
		
		Now we assume that there exists no common vertex $z_2$ in the adjacencies of $x$ and $y$ in $B_{2,1}$.
		Consider a shortest path $xPy$ from $x$ to $y$ where all the intermediate vertices are in $B_2$.
		Let $x^\prime$ and $y^\prime$ be the vertices adjacent to $x$ and $y$ respectively in the path.
		Clearly there cannot be an edge from $x$ or $y$ to the intermediate vertices of $x P y$ since otherwise $xPy$ cannot be a shortest path.
		Thus only possible edges in order to avoid bigger cycles due to the cocomparability property and the interval nature of the graph $G[B_1 \cup B_2]$, are from $z_1$ to the intermediate vertices of the path $x P y$. 
		Hence $z_1x^\prime$ and $z_1y^\prime$ edges must exist. As in the previous paragraph, an edge $x^\prime y^\prime$ also exists. 
		This forms a diamond diagonal $z_1x^\prime$ of the diamond induced by the set $\{x,x^\prime,y^\prime,z_1\}$ which is a contradiction.
		
		Now suppose a vertex $x_1 \in B_{1,2}$ is not adjacent to a vertex $y_1 \in B_{2,1}$. Then
		let $y_2$ be a neighbor of $x_1$ in $B_{2,1}$ and $x_2$ be a neighbor of $y_1$
		in $B_{1,2}$.  The existence of $x_2$ and $y_2$ are guaranteed by the
		definition of $B_{i,j}$. Notice that since $G[B_{1,2}]$ and $G[B_{2,1}]$
		are cliques, $x_1$ is adjacent to $x_2$ and $y_1$ is adjacent to $y_2$.  
		If $x_2$ and $y_2$ are non-adjacent, $G[\{x_1, y_2, y_1, x_2\}]$ is an induced $C_4$ in $B_1 \cup B_2$ which is a contradiction.  If $x_2$ and $y_2$ are adjacent, the edge
		${x_2y_2}$ is the diagonal of the diamond $G[\{x_1, y_2, y_1, x_2\}]$ which is a contradiction.
		
		\bigskip \label{claimConsecutiveProof} \noindent \textbf{Proof of Claim 2}:  First we need to prove that any induced cycle $C_n, n \geq 3$ in $R_G$ ensures the existence of an induced cycle $C_m$ in $G$ where $m \geq n$. 			
		Consider any induced cycle $C_n = \{B_0, \ldots, B_{n-1}\}$ in $R_G$. 
		Let $G_n$ be the induced subgraph $G[B_0 \cup \cdots \cup B_{n-1}]$.
		For any $i \in \{0, \ldots, n-1\}$, if $B_{i,i-1} \cap B_{i,i+1} \neq \emptyset$ (addition and subtraction are $modulo$ $n$), choose any vertex $b_{i}$ in that intersection. 
		On the contrary, if $B_{i,i-1} \cap B_{i,i+1} = \emptyset$, then consider the vertices $b_{i,i-1} \in B_{i,i-1}$ and $b_{i,i+1} \in B_{i,i+1}$ such that the distance between them is the shortest. Thus there exists a shortest path $b_{i,i-1} P b_{i,i+1}$ which cannot contain an intermediate vertex in the  sets $B_{i,i-1}$ and $B_{i,i+1}$. 		
		By Claim \ref{claimClique}, $B_{i,i-1} \cup B_{i-1,i}$ induces a clique. Similarly $B_{i,i+1} \cup B_{i+1,i}$ also induces a clique. This guarantees the adjacency between the chosen vertices of any two consecutive branch sets. 
		Thus we get a cycle $C_m$ in $G$ in which each branch set $B_{i}$ contributes either a vertex $b_{i}$ or a path $b_{i,i-1} P b_{i,i+1}$.
		Clearly there is no chord in $C_m$ between two vertices belonging to two non-adjacent branch sets in $C_n$ since it becomes a chord in $C_n$ too.
		Similarly there is no chord in $C_m$ between two vertices belonging to two adjacent branch sets, say $B_i$ and $B_{i+1}$, in $C_n$ since we select exactly one vertex each from $B_{i,i+1}$ and $B_{i+1,i}$.
		Thus if there exists an $i$ with condition $B_{i,i-1} \cap B_{i,i+1} = \emptyset$, then we get $m > n$. Otherwise $m = n$.
		
		Now suppose $C$ is an induced cycle containing $B_0$, $B_1$ and $B_2$ as consecutive vertices. 
		Clearly $C$ can be either $C_3$ or $C_4$ since $R_G$ is a cocomparability graph.
		Assume it is a $C_3$. Then $\{B_0,B_1,B_2\}$ induces $C$.
		By the discussion in the previous paragraph, if $B_{1,0} \cap B_{1,2} = \emptyset$, then there exists a bigger induced cycle $C_m$ in $G$. The only possible bigger induced cycle is a $C_4$ which contradicts the fact that two vertices of an induced $C_4$ resides inside a branch set. 	
		Now assume $C$ is a $C_4$, that is $\{B_0,B_1,B_2,B_3\}$ is an induced $C_4$.
		By the first step, if $B_{1,0} \cap B_{1,2} = \emptyset$, then there exists a bigger induced cycle $C_m$ in $G$, $m>4$. This is a contradiction.
		
		Thus by picking $b_i \in B_{i,i-1} \cap B_{i,i+1}$ $(0\leq i \leq k-1)$, we get the 
		an induced $k$-cycle in $G$ with $b_i \in B_i$.
		
		\bigskip \label{claimBipartiteproof} \noindent \textbf{Proof of Claim 3}: First we need to prove the following sub-claims.
		\begin{enumerate}[label=(3.\roman*)]
			\item \label{triangleClique} \noindent If $B_0, B_1$, $B_2$ induce a triangle in $R_G$ then 
			$B_{i,i-1} = B_{i,i+1}$ $(0\leq i\leq2)$ where addition and subtraction are modulo $3$ and hence the set $B_{0,1} \cup B_{1,2} \cup B_{2,0}$
			induces a clique in $G$. 
			
			\textit{Proof:} Using Claim \ref{claimConsecutive}, choose the vertices $ b_0 \in B_{0,1} \cap B_{0,2}$, $ b_1 \in B_{1,0} \cap B_{1,2}$ and $ b_2 \in B_{2,0} \cap B_{2,1}$.
			Clearly $\{b_0,b_1,b_2\}$ induces a triangle. 
			Let $i=1$ without loss of generality. Assume $B_{1,0} \neq B_{1,2}$. Let $B_{1,0} \setminus B_{1,2} \neq \emptyset$ without loss of generality. 
			Take a vertex $b_1^\prime \in B_{1,0} \setminus B_{1,2}$.
			Since $B_{1,0} \cup B_{0,1}$ is a clique, $b_1^\prime$ is adjacent to $b_0$ and $b_1$.
			Moreover since $b_1^\prime \notin B_{1,2}$, it is not adjacent to $b_2$. Hence $\{b_0,b_1,b_2,b_1^\prime\}$ induces a diamond with the diamond-diagonal $b_0b_1$ between $B_0$ and $B_1$ which is a contradiction. $\blacksquare$
			
			\item \label{underlyingC4} \noindent If $\{b_0,b_1,b_2,b_3\}$ is an induced $C_4$ in $G$, then 
			$\{B_0,B_1,B_2,B_3\}$ is an induced $C_4$ in $R_G$, where for each $i$,
			$B_i$ is the branch set containing $b_i$.
			
			\textit{Proof:} 	By the definition of $R_G$, it is clear that the branch sets $B_0,B_1,B_2,B_3$ are all distinct and
			$B_0B_1,B_1B_2,B_2B_3,B_1B_3$ are edges in $R_G$.
			Assume that $\{B_0,B_1,B_2,B_3\}$ is not an induced $C_4$. 
			Without loss of generality, if $B_0B_2$ is an edge in $R_G$, then $R_G$ has a triangle $\{B_0,B_1,B_2\}$.
			By the sub-claim \ref{triangleClique}, $B_{0,1} = B_{0,2}$ and $B_{2,0} = B_{2,1}$. Hence $b_0 \in B_{0,2}$ and $b_2 \in B_{2,0}$.
			By Claim \ref{claimClique}, $b_0b_2$ edge also exists in $G$ which contradicts the fact that $\{b_0,b_1,b_2,b_3\}$ is an induced $C_4$.
			Hence $\{B_0,B_1,B_2,B_3\}$ is also an induced $C_4$. $\blacksquare$
			
			\item \label{triangleR_G} \noindent Every edge of a triangle in $R_G$ is also part of an induced $C_4$ in $R_G$.
			
			\textit{Proof:} By the minimality of $R_G$, any edge is contracted if the resulting branch set does not contain two vertices of an induced $C_4$ of $G$.
			Thus if there exists a triangle in $R_G$, every pair of branch sets constituting the triangle contains two vertices of an induced $C_4$ of $G$.
			Also by the sub-claim \ref{underlyingC4}, each such pair of branch sets should be adjacent vertices of an induced $C_4$ of $R_G$.
			Hence the statement. $\blacksquare$ 
			
			\item \label{triangleFree} \noindent $R_G$ is a triangle-free.
			
			\textit{Proof:} By the previous sub-claim, if there exists a triangle $R_3$ in $R_G$ induced by the set $\{B_0,B_1,B_2\}$, each edge of $R_3$ will also be part of an induced $C_4$ of $R_G$. Consider an induced subgraph $R_M$ of $R_G$ with \emph{minimum} number of vertices such that it contains $V(R_3)$ and branch sets corresponding to the induced $4$-cycles sharing edge with $R_3$ (Figure \ref{fig:triangleC_4}). 
			Each induced $C_4$ sharing edge $B_iB_{i+1}$ is denoted as $\{B_i,B_{i+1}, B_{{i+1}^\prime},B_{i^{\prime\prime}}\}$ where $i\in \{0,1,2\}$ and addition is modulo $3$.
			We do not claim the adjacencies or non-adjacencies among the pair $\{B_{i^\prime},B_{i^{\prime\prime}}\}$.
			Clearly any vertex in the set $B_{i^{\prime\prime}}$ (respectively $B_{i^\prime}$) is non-adjacent to the a vertex in $B_{i+1}$ (respectively $B_{i+2}$) since they are non adjacent vertices of an induced $C_4$. 
			Similarly any vertex in the set $B_{i^{\prime\prime}}$ (respectively $B_{i^\prime}$) is non-adjacent to the a vertex in $B_{i+2}$ (respectively $B_{i+1}$) since otherwise it causes the existence of a diamond diagonal across two branch sets.			
			For any $i$, if $B_{i^\prime} \neq B_{i^{\prime\prime}}$, then $B_{i^\prime}$ must be non-adjacent to $B_{{i+1}^\prime}$ and $B_{i^{\prime\prime}}$ must be non-adjacent to $B_{{i+2}^{\prime\prime}}$ since we consider the minimal graph $R_M$.
			
			Now in the minimal graph $R_M$, if $B_{i^\prime} \neq B_{i^{\prime\prime}}$ for all $i\in \{0,1,2\}$, then there exists an \emph{asteroidal triple} $\{B_{0^\prime},B_{1^\prime},B_{2^\prime}\}$ in $R_M$ (hence in $R_G$) which contradicts the fact that $R_G$ is a cocomparability graph.
			If exactly two such pairs are unequal, without loss of generality $B_{0^\prime} \neq B_{0^{\prime\prime}}$ and $B_{2^\prime} \neq B_{2^{\prime\prime}}$, then there exists a $5$-cycle induced by the set $\{B_{0^{\prime\prime}},B_0,B_2,B_{2^\prime},B_{1^\prime}\}$ in $R_M$ (hence in $R_G$) which is a contradiction. 
			If exactly one such pair is unequal, without loss of generality $B_{0^\prime} \neq B_{0^{\prime\prime}}$, then there exists the same $5$-cycle in $R_M$ which is again a contradiction.
			If all such pairs are equal, then the set $\{B_{0^\prime},B_{1^\prime},B_{2^\prime}\}$ induces a triangle.
			As per Sub-claim \ref{triangleClique}, $B_{0^\prime,1^\prime} \cup B_{1^\prime,2^\prime} \cup B_{2^\prime,0^\prime}$ induces a clique. Similarly $B_{0,1} \cup B_{1,2} \cup B_{2,0}$ also induces a clique. 
			For $i \in \{0,1,2\}$, choose the vertices $b_i \in B_{i,i^{\prime}} \cap B_{i,i+1}$ which is nonempty by Claim $(\ref{claimConsecutive})$ and $b_{i^{\prime}} \in B_{i^{\prime},i} \cap B_{i^\prime, {i+1}^\prime}$ which is also nonempty by Claim $(\ref{claimConsecutive})$.
			One can verify that $G[\{b_0,b_1,b_2, b_{0^\prime}, b_{1^\prime}, b_{2^\prime}\}]$ is an induced $\overline{C_6}$ in $G$ which is a contradiction since $G$ satisfies the conditions of Theorem \ref{theorem:vpg}. Hence $R_G$ is triangle free. $\blacksquare$
		\end{enumerate}
		
		Thus by the above sub-claim, $R_G$ is a triangle-free graph. Also $R_G$ contains no bigger induced odd cycles since it is a cocomparability graph. Thus $R_G$ is a bipartite graph and hence a comparability graph.
		$R_G$ is also a permutation graph since it is both cocomparability and comparability graph \cite{pnueli1971transitive}. Hence the claim.		
	\end{proof}
	
	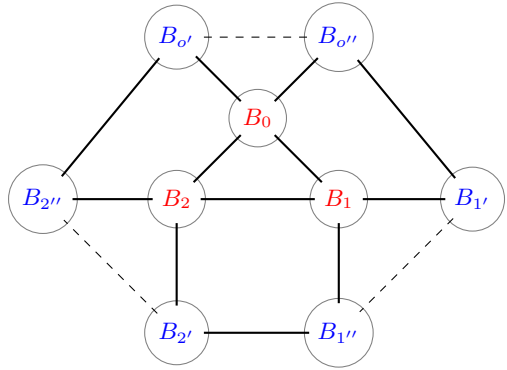
\begin{figure}
		\centering
		\begin{tikzpicture}[scale=.7, roundnode/.style={circle, draw=black!50, minimum size=4mm}]
		
		\node[roundnode] (B0) at (0,0) {\color{red}$B_0$};
		\node[roundnode] [below left of=B0, node distance=.6in] (B2) {\color{red}$B_2$};
		\node[roundnode] [below right of=B0, node distance=.6in] (B1) {\color{red}$B_1$};

		\node[roundnode] [above left of=B0, node distance=.6in] (B01) {\color{blue}$B_{o^\prime}$};
		\node[roundnode] [above right of=B0, node distance=.6in] (B02) {\color{blue}$B_{o^{\prime\prime}}$};
		
		\node[roundnode] [right of=B1, node distance=.7in] (B11) {\color{blue}$B_{1^\prime}$};
		\node[roundnode] [below of=B1, node distance=.7in] (B12) {\color{blue}$B_{1^{\prime\prime}}$};
		
		\node[roundnode] [left of=B2, node distance=.7in] (B22) {\color{blue}$B_{2^{\prime\prime}}$};
		\node[roundnode] [below of=B2, node distance=.7in] (B21) {\color{blue}$B_{2^\prime}$};

		\draw [black, thick, shorten <= -2pt, shorten >=-2pt] (B0) -- (B1);
		\draw [black, thick, shorten <= -2pt, shorten >=-2pt] (B0) -- (B2);
		\draw [black, thick, shorten <= -2pt, shorten >=-2pt] (B1) -- (B2);	 
		
		\draw [black, thick, shorten <= -2pt, shorten >=-2pt] (B0) -- (B01);
		\draw [black, thick, shorten <= -2pt, shorten >=-2pt] (B0) -- (B02);
		\draw [black, dashed, shorten <= -2pt, shorten >=-2pt] (B01) -- (B02);	 
		
		\draw [black, thick, shorten <= -2pt, shorten >=-2pt] (B1) -- (B11);
		\draw [black, thick, shorten <= -2pt, shorten >=-2pt] (B1) -- (B12);
		\draw [black, dashed, shorten <= -2pt, shorten >=-2pt] (B11) -- (B12);	
		
		\draw [black, thick, shorten <= -2pt, shorten >=-2pt] (B2) -- (B21);
		\draw [black, thick, shorten <= -2pt, shorten >=-2pt] (B2) -- (B22);
		\draw [black, dashed, shorten <= -2pt, shorten >=-2pt] (B21) -- (B22);	
		
		\draw [black, thick, shorten <= -2pt, shorten >=-2pt] (B02) -- (B11);
		\draw [black, thick, shorten <= -2pt, shorten >=-2pt] (B12) -- (B21);
		\draw [black, thick, shorten <= -2pt, shorten >=-2pt] (B22) -- (B01);
		\end{tikzpicture}
		\caption{Minimal graph $R_M$ having a $C_3$ with each edge part of an induced $C_4$.}
		\label{fig:triangleC_4}
	\end{figure}
	
	We relabeled the branch sets of the bipartite permutation graph $R_G$ in subsection \ref{subsectAlgo}. The following two lemmas are relevant to this labeling.
	
	\bigskip  \noindent \textbf{Lemma }\ref{lemma:branchContinuity} Statement : \emph{For any two branch sets $B_i,B_j$ of the same parity if $i<j$, then $B_i \prec_\sigma B_j$.}
	\begin{proof}
		Every branch set of $R_G$ is a connected graph since every edge contraction towards the formation of $R_G$ assures connectivity.
		Suppose $\exists x \in B_i$  and $\exists y \in B_j$ such that $y \prec_\sigma x$.
		Then any path in $G[B_i]$ from the leftmost vertex, say $z$, of $B_i$ to $x$ does not contain a vertex from the closed neighborhood of $y$. 
		Hence the triple $(z,y,x)$ is a forbidden triple in $\sigma$ which is a contradiction by Lemma \ref{lemma:forbiddentriple_contradiction}.
	\end{proof}
	
	\begin{lemma}
		\label{lemma:convexR_G}
		The relabeling of $V(R_G)$ respects the \emph{adjacency property}.
	\end{lemma}
	\begin{proof}
		We know that $\sigma_{R_G}$ is an umbrella-free ordering. Moreover, the relabeling ensures that the labels in each part of $R_G$ inherit the order of $\sigma$. 
		In order to prove the \emph{adjacency property} of the relabeling of $V(R_G)$, 
		we assume that there are three branch sets $B_j,B_k,B_l$ $(j < k < l)$ of the same parity and a fourth branch set $B_i$ of the opposite parity adjacent to $B_j$ and $B_l$, but not $B_k$. Then either $\left(B_i \prec_{\sigma_{R_G}} B_k \prec_{\sigma_{R_G}} B_l\right)$ or $\left(B_j \prec_{\sigma_{R_G}} B_k \prec_{\sigma_{R_G}} B_i\right)$ forms an umbrella which contradicts the umbrella-freeness of $\sigma_{R_G}$.
	\end{proof}	
	\subsection{Step 2. Interval representations of branch-sets of $R_G$}
	
	The \emph{closure} of a branch set $B_i$ of $R_G$ is the set $B_i^* = B_i
	\cup \{B_{j,i} : B_iB_j \in E(R_G)\}$. 
	That is, $B_i^*$ consists of all the vertices in $B_i$ together with their
	neighbors in $G$.
	
	In order to fit together interval representations of the individual
	branch-sets, it turns out to be necessary to start with interval
	representations for the closure of each branch-set.
	We also need these interval representations to satisfy the properties
	listed in the following lemma.

	\bigskip \noindent \textbf{Lemma \ref{lemma:interval_R_G}} Statement: 
		\emph{For each branch set $B_i$ of $R_G$, $G[B_i^*]$ is an interval graph.
		Moreover, $G[B_i^*]$ has an interval representation $\mathcal{I}_i^*$
		satisfying	the following properties.
		\begin{enumerate}[label=(\roman*)]
			\item
			For all $x, y \in B_i^*$, we have the interval for $x$ to the left
			of interval for $y$ if and only if $x \prec_P y$.
			\item 
			For each neighbor $B_j$ of $B_i$, all the intervals corresponding
			to the vertices in $B_{j,i}$ are point intervals at a point
			$p_{j,i}$.
			\item 
			If $B_j$ and $B_k$ are two neighbors of $B_i$ such that $j<k$, then
			the point $p_{j,i}$ is to the left of the point $p_{k,i}$ in
			$\mathcal{I}_i^*$.
		\end{enumerate}}
	\begin{proof}
		By Lemma \ref{lemma:R_G}.\ref{claimBipartite}, $R_G$ is a bipartite
		permutation graph.
		If $G[B_i^*]$ contains an induced $C_4$, then by Lemma
		\ref{lemma:R_G}.\ref{underlyingC4} (sub-claim in Appendix \ref{appAlgo}), $R_G$ contains	an induced $C_4$ in the closed neighborhood of $B_i$.
		But since $R_G$ is bipartite, the closed neighborhood of any vertex is
		a star.
		Hence $G[B_i^*]$ is a $C_4$-free cocomparability graph and thus an
		interval graph \cite{gilmore1964characterization}.
		
		\bigskip \noindent $(i)$: Let $P_G(V(G),\prec_P)$ be the poset assumed in Algorithm \ref{algo:\B0drawing} and $P_i^* = (B_i^*,\prec_i^*)$ be the subposet of $P_G$ induced on $B_i^*$.  
		Clearly $P_i^*$ is $(2+2)$-free since $G[B_i^*]$ is $C_4$-free.  
		Hence $P_i^*$ is an interval order (Theorem~\ref{theorem:fishburn}).
		Pick any interval representation of the interval order $P_i^*$ and
		consider it as an interval representation $\mathcal{I}_i^*$ for the
		interval graph $G[B_i^*]$. 
		Thus if $I_x$ is to the left of $I_y$ in $\mathcal{I}_i^*$, that is $I_x \prec_{\mathcal{I}_i^*} I_y$, then $x \prec_i^* y$ and hence $x \prec_P y$. Similarly, if $x \prec_P y$, then $x \prec_i^* y$ and hence $I_x \prec_{\mathcal{I}_i^*} I_y$.
		
		\bigskip \noindent $(ii)$: 
		By Lemma \ref{lemma:R_G}.\ref{claimClique}, $B_{i,j} \cup B_{j,i}$ is a
		clique.
		As the vertices corresponding to the clique share a common point in the
		interval representation $\mathcal{I}_i^*$, there exists a point
		$p_{j,i}$ common to all the intervals corresponding to the vertices in
		the set $B_{j,i} \cup B_{i,j}$.
		The point $p_{j,i}$ is not contained in an interval corresponding to
		any other vertex in $B_i$ since they are non-adjacent to the vertices
		in $B_{j,i}$.  
		Since the vertices in $B_{j,i}$ have no neighbors in $G[B_i^*]$ other
		than $B_{j,i} \cup B_{i,j}$, we can shrink all the intervals
		corresponding to them to the single point $p_{j,i}$ in
		$\mathcal{I}_i^*$.
		
		\bigskip \noindent $(iii)$: 
		Since $j$ and $k$ have the same parity and $j<k$, Lemma
		\ref{lemma:branchContinuity} guarantees that $\forall x \in B_j,
		\forall y \in B_k$, $x \prec_\sigma y$. Since $\sigma$ is the linear extension of $P_G$ and $x$ is non-adjacent to $y$ in $G$, $x \prec_P y$ which also implies that $x	\prec_i^* y$. Thus the point intervals at $p_{j,i}$ is to the left of the point intervals at $p_{k,i}$ in $\mathcal{I}_i^*$.
	\end{proof}
	We remove the intervals other than those corresponding to the vertices of
	$B_i$ from $\mathcal{I}_i^*$ to get an interval representation
	$\mathcal{I}_i$ of $G[B_i]$. 
	But we retain the location of the point $p_{j,i}$ corresponding to each
	neighbor $B_j$ of $B_i$.
	
	\subsection{Step 3. \B0 representation of $G$}
	
	In this final step, we draw a \B0 representation $D$ of $G$ by combining
	interval representations $\mathcal{I}_i$ of each $B_i$ in $V(R_G)$
	and a \B0 representation of $R_G$.
	
	The drawing procedure is rewritten here to maintain continuity for readers.	
	Let $e$ and $o$ respectively (with further subscripts if needed) denote the
	even and odd indices of the branch sets of $R_G$.
	
	\begin{enumerate}[label=(\roman*)]
		\item 
		For each odd-indexed branch set $B_o$, $\mathcal{I}_o$ is drawn
		vertically from the point $(o,(e_1-0.5))$ to $(o,(e_2+0.5))$ where
		$B_{e_1}$ and $B_{e_2}$ are the leftmost and rightmost neighbors of
		$B_o$ in $\sigma_{R_G}$.
		\item 
		Stretch or shrink the intervals in each vertical interval
		representation $\mathcal{I}_o$ without changing their intersection
		pattern, so that for each neighbor $B_e$ of $B_o$, the point $p_{e,o}$
		is at $(o,e)$.
		This can be done since the intersection pattern of an interval
		representation with $n$ intervals is solely determined by the order of
		the corresponding $2n$ endpoints.
		
		\item 
		For each even-indexed branch set $B_{e}$, $\mathcal{I}_e$ is drawn
		horizontally from the point $((o_1-0.5),e)$ to $((o_2+0.5),e)$ where
		$B_{o_1}$ and $B_{o_2}$ are the leftmost and rightmost neighbors of
		$B_e$ in $\sigma_{R_G}$.
		\item 
		Stretch or shrink the intervals in each horizontal interval
		representation $\mathcal{I}_{e}$ without changing their intersection
		pattern, so that for each neighbor $B_o$ of $B_e$, the point $p_{o,e}$
		is at $(o,e)$.
	\end{enumerate}
	\bigskip \noindent \textbf{Proposition \ref{proposeD}} Statement:
		\emph{The B$_0$-VPG representation $D$ is precisely a B$_0$-VPG
		representation of the cocomparability graph $G$.}
	\begin{proof}
		We need to verify the pairwise relationships of $G$ in $D$.  The
		adjacency and non-adjacency among the vertices inside each branch set
		$B_i$ are guaranteed since we use the interval representation $I_i$ in
		$D$.
		
		For two adjacent branch sets $B_o$ and $B_e$, the points $p_{o,e}$ and
		$p_{e,o}$ coincide at $(o,e)$. 
		Since  $B_{o,e} \cup B_{e,o}$ forms a clique in $G$ (Lemma
		\ref{lemma:R_G}.\ref{claimClique}), we need to show that the
		corresponding paths share a common point in $D$. 
		The point $p_{o,e}$ (respectively $p_{e,o}$) is the common point of the
		intervals corresponding to the vertices in $B_{o,e}$ (respectively
		$B_{e,o}$).  Since both these two points fall at location $(o,e)$
		in the plane, all the paths corresponding to vertices in $B_{o,e}
		\cup B_{e,o}$ have a common intersection.
		If two vertices $x \in B_{o}$ and $y \in B_{e}$ are non-adjacent,
		then the intervals $I_x \in \mathcal{I}_{o}, I_y \in \mathcal{I}_{e}$
		satisfy either $p_{e,o} \notin I_x $ or $p_{o,e} \notin I_y$ or both.
		Thus $I_x$ and $I_y$ do not touch each other in $D$.
		
		In $D$, if we consider each $\mathcal{I}_i$ as a single path,
		then we can use the adjacency property to verify that we get
		a \B0 representation of $R_G$. 
		Hence the non-adjacencies of the two vertices belonging to two
		non-adjacent branch sets of $R_G$ are maintained in $D$.
	\end{proof}
	
	This completes the proof of Theorem~\ref{theorem:vpg} since we were able to construct a \B0 representation $D$ for the cocomparability graph $G$ satisfying the conditions of Theorem \ref{theorem:vpg}.

\end{subappendices}

\end{document}